\documentclass[nofootinbib,superscriptaddress,floatfix,pra,table,aps, onecolumn]{revtex4-2}

\usepackage{scrextend}
\usepackage{soul}
\usepackage[pdftex]{graphicx}
\usepackage{dcolumn}
\usepackage{bm}
\usepackage{epsfig}
\usepackage{amsmath}
\usepackage{amssymb} 
\usepackage{amsfonts}
\usepackage{pifont} 
\usepackage{mathtools}
\usepackage{multirow}
\usepackage{float}
\usepackage[utf8]{inputenc} 
\usepackage{epstopdf}
\usepackage{dsfont}
\usepackage{MnSymbol}
\usepackage{xspace}
\usepackage[most]{tcolorbox}
\tcbuselibrary{breakable}
\usepackage{multirow}
\usepackage{enumitem}
\usepackage{moreenum}
\usepackage{csquotes}
\usepackage{tabularray}
\usepackage{booktabs}
\usepackage[linesnumbered,lined,boxed,commentsnumbered,ruled,vlined]{algorithm2e}
\usepackage{dashbox}

\usepackage[dvipsnames]{xcolor}

\usepackage{amsthm}
\newtheorem{theorem}{Theorem}
\newtheorem*{theorem*}{Theorem}
\newtheorem{lemma}[theorem]{Lemma}
\newtheorem*{lemma*}{Lemma}
\newtheorem{conjecture}[theorem]{Conjecture}
\newtheorem{corollary}[theorem]{Corollary}
\newtheorem*{corollary*}{Corollary}
\newtheorem{proposition}[theorem]{Proposition}
\newtheorem*{proposition*}{Proposition}
\newtheorem{definition}[theorem]{Definition}
\newtheorem*{definition*}{Definition}
\newtheorem{example}[theorem]{Example}
\newtheorem*{example*}{Example}

\newtcbtheorem[]{tcexample}{Example}{breakable, colback=orange!3!white,colframe=orange!85!black, fonttitle=\bfseries}{Exa}
\newtcolorbox[]{thmbox}[1][]{breakable,colback=NavyBlue!15!white,colframe=white,#1}

\DeclareMathOperator{\spn}{span}

\def\g{\mathfrak{g}}

\def\su{\mathfrak{su}}
\def\Z{\mathbb{Z}}
\def\N{\mathbb{N}}
\def\C{\mathbb{C}}
\def\R{\mathbb{R}}
\def\Q{\mathbb{Q}}

\renewcommand{\sl}[2]{\mathfrak{sl}({#1},{#2})}
\newcommand{\Np}[1]{\N^{#1}_{\geq 0}}

\newcommand{\x}[2]{\chi_{#1}(#2)}

\newcommand{\Tr}[1]{\mathrm{Tr}\,(#1)}
\newcommand{\iF}[2]{\langle{#1},{#2}\rangle_{\mathrm{IF}}}

\newcommand{\ket}[1]{\ensuremath{| #1 \rangle}{}}
\newcommand{\bra}[1]{\ensuremath{\langle #1 |}{}}

\let\oldcite\cite
\renewcommand{\cite}[1]{\textup{\oldcite{#1}}}

\usepackage{bbm}
\newcommand{\unity}{\mathbbmss{1}}

\usepackage[bookmarks=false,pdfstartview={FitH}]{hyperref}
 
\usepackage{tikz,xcolor}
\usetikzlibrary{patterns,patterns.meta,decorations.pathreplacing,angles,quotes, plotmarks, arrows.meta, calc, hobby}
\usetikzlibrary{shapes.geometric}
\usetikzlibrary{cd, babel, math} 
\definecolor{lime}{HTML}{A6CE39}
\DeclareRobustCommand{\orcidicon}{%
	\begin{tikzpicture}
	\draw[lime, fill=lime] (0,0) 
	circle [radius=0.16] 
	node[white] {{\fontfamily{qag}\selectfont \tiny ID}};
	\draw[white, fill=white] (-0.0625,0.095) 
	circle [radius=0.007];
	\end{tikzpicture}
	\hspace{-2mm}
}
\foreach \x in {A, ..., Z}{%
	\expandafter\xdef\csname orcid\x\endcsname{\noexpand\href{https://orcid.org/\csname orcidauthor\x\endcsname}{\noexpand\orcidicon}}
}

\begin{document}
\title{Finite-dimensional Lie subalgebras of the universal enveloping algebra of $\sl{2}{\C}$}
\date{\today}

\author{Tim Heib\,\orcidH{}}
\email{t.heib@fz-juelich.de}	
\affiliation{Institute for Quantum Computing Analytics (PGI-12), Forschungszentrum J\"ulich, 52425 J\"ulich, Germany}
\affiliation{Theoretical Physics, Universit\"at des Saarlandes, 66123 Saarbr\"ucken, Germany}

\begin{abstract}
    In this work we provide the complete classification of all non-zero finite-dimensional complex Lie subalgebras of the universal enveloping algebra of the special linear Lie algebra $\sl{2}{\C}$. That is, we show that every such nilpotent Lie subalgebra is abelian and every non-solvable Lie subalgebra is reductive, with semisimple component isomorphic to $\sl{2}{\C}$ itself. Every non-nilpotent solvable Lie subalgebra is the direct sum of an abelian center and a member of the family of solvable Lie algebras $\mathfrak{r}(\boldsymbol{d},\boldsymbol{m})$, parameterized by a vector $\boldsymbol{d}$ whose entries are strictly increasing positive integers that have greatest common divisor one together with a positive integer vector $\boldsymbol{m}$ of the same length recording their multiplicities.
    Moreover, we discuss the corresponding classification in a real skew-hermitian setting and its relation to the realizations in the single- and two-mode Weyl algebras.
\end{abstract}
\maketitle


\section*{Introduction}\label{sec:intro}

The endeavor of classifying finite-dimensional Lie algebras initially began with the work of Sophus Lie and was significantly advanced by the classification of all complex finite-dimensional semisimple Lie algebras, first established through the works of Wilhelm Killing and Élie Cartan~\cite{Lie:1880,Lie:1891,Lie:1893,Killing:zweiter:1888,Killing:dritter:1889,Killing:vierter:1890,Cartan:1894}. The Levi-Mal'tsev theorem later provided a structural foundation for the general case by establishing that every real or complex finite-dimensional Lie algebra can be decomposed into a semidirect sum of a semisimple Lie algebra and its radical~\cite{Levi:1905,Matltsev:1942}. In a next step, one consequently needs to classify all finite-dimensional solvable Lie algebras, as well as the possible actions of semisimple Lie algebras on them. Such a general classification has not yet been found. However, for low-dimensional cases, classifications of solvable Lie algebras have been established~\cite{Bianchi:1903,Mubarakzyanov:1963,Popovych:2003}.

A complementary approach is to determine which finite-dimensional Lie algebras admit faithful realizations within a prescribed infinite-dimensional Lie algebra. For example, all finite-dimensional complex Lie algebras that can be realized as Lie subalgebras of the single-mode Weyl algebra $A_1$ have been classified up to isomorphism~\cite{TST:2006}. In the real skew-hermitian Weyl algebra $\hat{A}_1$, corresponding classification results have been established for all nilpotent and non-solvable Lie subalgebras, as well as for all Lie subalgebras containing a free Hamiltonian~\cite{Heib:2025:A1}.

Related classification problems arise in Lie-algebraic approaches to molecular structure and dynamics~\cite{Levine:1982}. In this context, the search for Schrödinger operators in universal enveloping algebras of finite-dimensional Lie algebras of first-order differential operators motivated the Lie-algebra classification efforts of Refs.~\cite{Artemio:1992,Artemio:1992:2}.

Here, we are motivated by recent work on the dimensional classification of bosonic dynamical Lie algebras, and we aim to investigate the corresponding classification problem for the universal enveloping algebra of the special linear Lie algebra $\sl{2}{\C}$~\cite{Bruschi:Xuereb:2024,Heib:2025:A1,Heib:2026:Main}. Thus, in this work, we consider the unital associative universal enveloping algebra $\mathcal{U}=U(\sl{2}{\C})$, equipped with its commutator bracket, and classify all finite-dimensional Lie subalgebras contained within it, up to isomorphism class. To be precise, we show the following classification theorem.

\begin{thmbox}[boxsep=0pt,left=1em,right=1em]{}
    \begin{theorem}[Classification result]\label{thm:classification:Usl2C}
    Let $\g$ be a non-zero finite-dimensional Lie subalgebra of $\mathcal{U}\equiv U(\sl{2}{\C})$. Then, $\g$ is isomorphic to exactly one of the following Lie algebras:
    \begin{itemize}
        \item If nilpotent: $\g\cong \C^d$ for some $d\in\N_{>0}$.
        \item If semisimple: $\g\cong\sl{2}{\C}$.
        \item If non-semisimple and non-solvable: $\g\cong\sl{2}{\C}\oplus \C^d$ for some $d\in\N_{>0}$.
        \item If solvable non-nilpotent: $\g\cong \mathfrak{r}(\boldsymbol{d},\boldsymbol{m})\oplus \C^{d_0}$ for some $\boldsymbol{d},\boldsymbol{m}\in\N_{>0}^{m}$, $m\in\N_{>0}$, and $d_0\in\N_{\geq0}$, such that $d_1<d_2<\ldots<d_m$ and $\operatorname{gcd}(d_1,\ldots,d_m)=1$.
    \end{itemize}
    \end{theorem}
\end{thmbox}

In Section~\ref{sec:def:and:conventions}, we recall the essential definitions, conventions, notation, and structural lemmata. In Section~\ref{sec:main:proof}, we provide the proof of the classification theorem above, by splitting the discussion into the nilpotent, semisimple, non-solvable, and non-nilpotent solvable settings, where we show that any such Lie subalgebra of $\mathcal{U}$ must be isomorphic to one of the classes stated above. Then, we show that, for each case, there exists at least one faithful realization as a Lie subalgebra of $\mathcal{U}$. We equip $\mathcal{U}$, furthermore, with a conjugate-linear anti-involution allowing us to investigate a real skew-hermitian case. In Section~\ref{sec:discussion:applications}, we discuss the consequences of the classifications of all finite-dimensional Lie subalgebras of $\mathcal{U}$ with respect to realizations of these Lie algebras within the single- and two-mode Weyl algebras, before ending with a conclusion in Section~\ref{sec:conclusion}.

\section{Definitions and conventions}\label{sec:def:and:conventions}

Unless explicitly stated otherwise, all vector spaces, Lie algebras and linear maps are taken over $\C$. Thus, the symbol \enquote{$\spn$} denotes the complex linear span, and we write \enquote{$\spn_{\mathbb{F}}$}, if it is the span over another field such as the real numbers. We classify Lie algebras up to Lie-algebra isomorphism, denoted by \enquote{$\cong$}.

\subsection{Lie-algebra notation}\label{subsec:Lie:algebra:notation}

Let us briefly recall the central concepts of this work.


\begin{definition}[Lie algebra]
    A Lie algebra $\g$ is a vector space equipped with a bilinear skew-symmetric bracket operation $[\cdot,\cdot]:\g\times\g\to\g$, called the Lie bracket, that satisfies the Jacobi identity
    \begin{align*}
        [x,[y,z]]+[y,[z,x]]+[z,[x,y]]=0\qquad\text{for all }x,y,z\in\g.
    \end{align*}
\end{definition}

For subspaces $V,W\subseteq\g$, we write $[V,W]:=\spn\{[v,w]\,\mid\, v\in V,\,w\in W\}$. A Lie subalgebra $\mathfrak{h}\subseteq\g$ is a linear subspace closed under the Lie bracket of $\g$; an ideal $\mathfrak{i}\subseteq\g$ satisfies $[\g,\mathfrak{i}]\subseteq \mathfrak{i}$. The quotient $\g/\mathfrak{i}$ has the Lie bracket $[x+\mathfrak{i},y+\mathfrak{i}]=[x,y]+\mathfrak{i}$, with zero element $\Bar{0}:=0+\mathfrak{i}$.

We further use the following definitions.


\begin{definition}\label{def:lie:algebra:concepts}
     Let $\g$ and $\mathfrak{h}$ be two Lie algebras and $V$ be a vector space. Then we define:
     \begin{enumerate}[label = (\alph*)]
         \item \emph{Lie-algebra homomorphism}. A Lie-algebra homomorphism $\varphi:\g\to \mathfrak{h}$ is a linear map that preserves the bracket structure, i.e., satisfies $\varphi([x,y])=[\varphi(x),\varphi(y)]$ for all $x,y\in \g$. It is an isomorphism if it is bijective, and an embedding if it is injective. In the latter case, we write $\varphi:\g\hookrightarrow\mathfrak{h}$.
         \item \emph{Representation}. A Lie-algebra homomorphism $\rho:\g\to\operatorname{End}(V)$, where $\operatorname{End}(V)$ is equipped with the commutator bracket, is called a representation. If $\rho$ is injective it is called faithful.

         The adjoint representation is the map $\operatorname{ad}_{(\cdot)}:\g\to \operatorname{End}(\g)$, where, for an element $x\in \g$, it is defined by $\operatorname{ad}_x(y):=[x,y]$.
         \item \emph{Centralizer and center}. Let $S\subseteq \g$. The centralizer of $S$ in $\g$ is the set $C_\g(S):=\{x\in\g\,\mid\,[x,s]=0\,\text{for all}\,s\in S\}$. We call the centralizer of $\g$ in $\g$ itself the center and denote it by $\mathcal{Z}(\g)=C_\g(\g)$. For a singleton $S=\{x\}$, we write $C_\g(x)\equiv C_\g(\{x\})$.
         \item \emph{Derived series and solvability}. The derived series of $\g$ is recursively defined via $\mathcal{D}^0\g:=\g$, $\mathcal{D}^{\ell+1}\g:=[\mathcal{D}^\ell\g,\mathcal{D}^\ell\g]$ for $\ell\in\Np{}$. The Lie algebra is solvable if the derived series terminates, i.e., if there exists some $k\in\Np{}$ such that $\mathcal{D}^k\g=\{0\}$. We abbreviate $\mathcal{D}\g\equiv\mathcal{D}^1\g=[\g,\g]$ for the first derived algebra. 
         \item \emph{Lower central series and nilpotency}. The lower central series of $\g$ is recursively defined via $\mathcal{C}^0\g:=\g$, $\mathcal{C}^{\ell+1}\g:=[\g,\mathcal{C}^\ell\g]$ for $\ell\in\Np{}$. The Lie algebra is nilpotent if this series terminates, i.e., if there exists some $k\in\Np{}$ such that $\mathcal{C}^k\g=\{0\}$.
     \end{enumerate}
\end{definition}

Moreover, for a finite-dimensional Lie algebra $\g$, we recall that the radical of $\g$ is its largest solvable ideal; $\g$ is simple if it is non-abelian and it has no proper nonzero ideals, while it is semisimple if its radical is trivial. We, furthermore, write $\C^d$ for the $d$-dimensional abelian Lie algebra, with the convention $\C^0\equiv\{0\}$.

We employ the method of associating labeled directed graphs with a finite-dimensional Lie algebra, developed in~\cite{Heib:2026:Graphs}, to illustrate the Lie algebras of interest and provide some insight into their internal structure. In each illustration, the lengths of arrows between distinct vertices (no loops) are proportional to the magnitude of their structure constants.


\begin{definition}[associated graph]\label{def:assocaited:graph}
    A finite-dimensional Lie algebra is minimal-graph-admissible if and only if it admits a basis $\mathcal{B}$ such that every nonzero bracket between two basis elements is a scalar multiple of one basis element. That is, for all $x,y\in\mathcal{B}$ with $[x,y]\neq0$, one has $[x,y]=\kappa z$, where $\kappa\in\C$ and $z\in\mathcal{B}$. For such a basis, the associated labeled directed graph has vertex set $V=\mathcal{B}$ and, for each $[u,v]=\kappa w\neq 0$, a directed edge from $u$ to $w$ labeled by $v$.
\end{definition}

\subsection{The universal enveloping algebra}\label{subsec:universal:enveloping:algebra}


\begin{definition}[Universal enveloping algebra]
    The universal enveloping algebra $U(\g)$ of a finite-dimensional Lie algebra $\g$ is the unital associative quotient of the tensor algebra $\bigoplus_{k\in\Np{}}\g^{\otimes k}$ by the two-sided ideal generated by
    \begin{align*}
        x\otimes y-y\otimes x-[x,y],\qquad\text{for }x,y\in\g.
    \end{align*}
\end{definition}

Equipped with the commutator bracket, the universal enveloping algebra $U(\g)$ is also a Lie algebra, which is infinite-dimensional whenever $\g$ is non-zero. Moreover, given a basis $\mathcal{B}=\{b_1,\ldots,b_n\}$ and a chosen preferred ordering of these basis elements, normal ordered monomials $b_1^{k_1}\ldots b_n^{k_n}$, with $(k_1,\ldots,k_n)\in\Np{n}$, form, according to the Poincaré--Birkhoff--Witt (PBW) theorem~\cite{PBW:Thm}, a basis of $U(\g)$.

In the following we consider $\mathcal{U}:=U(\sl{2}{\C})$, where $\sl{2}{\C}$ is the special linear Lie algebra, which is isomorphic to the matrix Lie algebra of all traceless $2\times 2$ matrices. In this work, we use the non-standard basis $\{X,Y,Z\}$ of $\sl{2}{\C}$ defined by
\begin{align*}
    [X,Y]&:=Z,\;&\;[Y,Z]&:=X,\;&\;[Z,X]&:=Y,
\end{align*}
whose real span is the compact real form $\su_\R(2)$ of $\sl{2}{\C}$~\cite{Pfeifer:2003,helgason2024differential}. This is motivated by the origin of this question stemming from the classification of finite-dimensional subalgebras of specific centralizer algebras of the skew-hermitian Weyl algebra $\hat{A}_n$~\cite{Heib:2026:Main}. Compared with the standard basis $\{H,E,F\}$, with $[H,E]=2E$, $[H,F]=-2F$, and $[E,F]=H$, used commonly throughout the literature~\cite{Kirillov:2008}, our convention can be obtained by the following identification:
\begin{align*}
    X&=\frac{1}{2}\left(E-F\right),\;&\; Y&=\frac{\mathrm{i}}{2}\left(E+F\right),\;&\;Z&=\frac{\mathrm{i}}{2}H.
\end{align*}
Illustrations of the $\sl{2}{\C}$ Lie algebra, based on both bases, can be found in Fig.~\ref{fig:sl2C:XYZ:basis}.

\begin{figure}[htpb]
    \centering
    \includegraphics[width=0.7\linewidth]{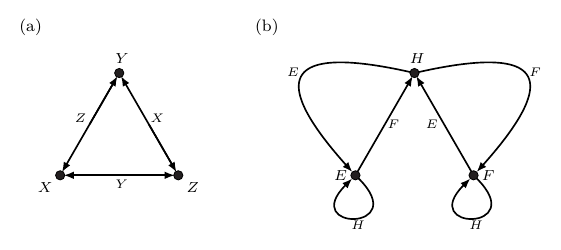}
    \caption{Illustration of the Lie algebra $\sl{2}{\C}$ using the basis $\{X,Y,Z\}$ in panel (a), whereas the more standard basis $\{H,E,F\}$ yields the graph depicted in panel (b), as per Definition~\ref{def:assocaited:graph}.}
    \label{fig:sl2C:XYZ:basis}
\end{figure}

By Racah's theorem, it holds that $\mathcal{Z}(\mathcal{U})=\C[C]$, where $C=X^2+Y^2+Z^2$, since $\sl{2}{\C}$ is a simple rank-one Lie algebra~\cite{Gould:1980,Greiner:1994,Bump:2013}. Generally, elements from the center of the universal enveloping algebra are called Casimir elements~\cite{Gould:1980}. Nevertheless, we will refer to $C$ as \enquote{the quadratic Casimir element}, or simply \enquote{the Casimir element}.

Unless stated otherwise, degree in the $\mathcal{U}$ arguments means the PBW degree: if $u=\sum_{a,b,c}\alpha_{abc}X^aY^bZ^c\neq 0$, then $\deg(u)=\max\{a+b+c\,\mid\,\alpha_{abc}\neq0\}$, and we set $\deg(0)=-\infty$, in agreement with the literature~\cite{Fischer:2005}. Moreover, we write $\unity$ for the associative unit. Products of elements in a finite-dimensional Lie subalgebra of $\mathcal{U}$ are taken in $\mathcal{U}$. If $p\in\C[a,b]$, then $p(C,x)$ denotes evaluation at the pair $(C,x)\in\mathcal{U}\times\mathcal{U}$.

\subsection{Further key definitions and conventions}\label{subsec:further:definitions}

Apart from the simple special linear Lie algebra $\sl{2}{\C}$, we must also introduce the family of solvable non-nilpotent Lie algebras relevant for the main classification result.


\begin{definition}\label{def:rnm}
    The Lie algebra $\mathfrak{r}(\boldsymbol{d},\boldsymbol{m})$ with $\boldsymbol{d},\boldsymbol{m}\in \N_{>0}^{m}$ for some $m\in\N_{>0}$, where the components $d_j$ are pairwise distinct, is defined  as the Lie algebra with basis $\{x, y_{jk_j}\,\mid\,j\in \{1,\ldots,m\},\, k_j\in\{1,\ldots,m_j\}\}$ satisfying the following Lie bracket relations
    \begin{align*}
        [x, y_{jk_j}]&=d_j y_{jk_j},\;&\;[y_{jk_j}, y_{\ell k_\ell}]&=0,
    \end{align*}
    for all admissible indices.
\end{definition}

It is straightforward to verify that the brackets in Definition~\ref{def:rnm} satisfy the Jacobi identity and therefore define a Lie algebra. An exemplary illustration of such a Lie algebra can be found in Fig.~\ref{fig:rdm}.

\begin{figure}[htpb]
    \centering
    \includegraphics[width=0.75\linewidth]{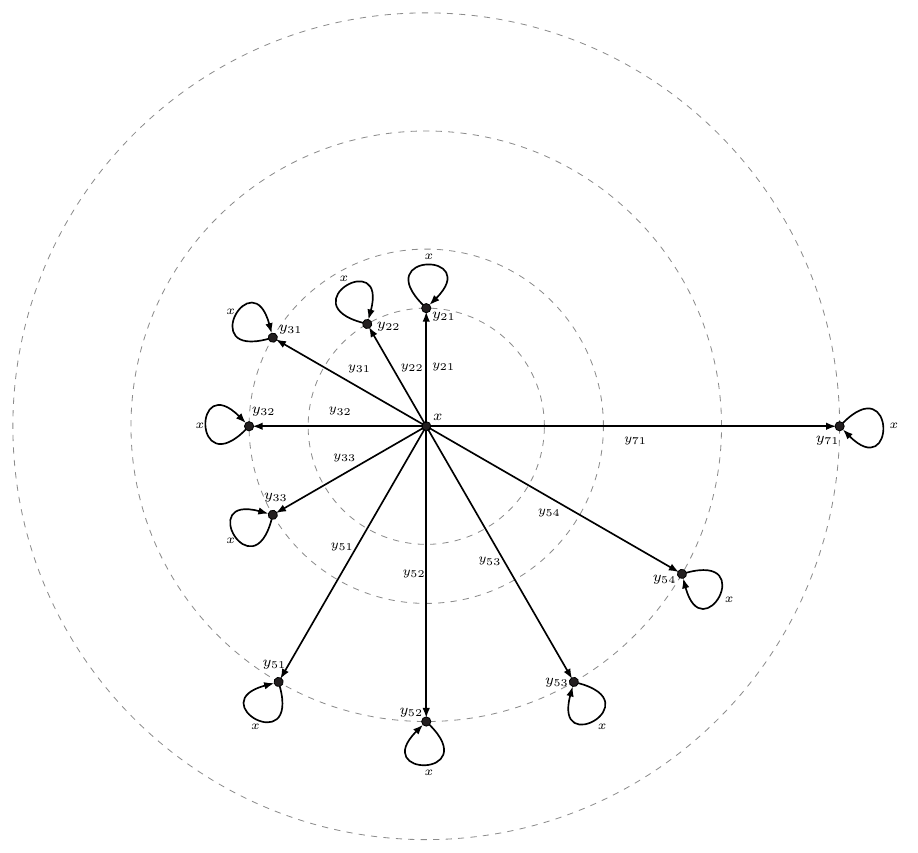}
    \caption{An exemplary illustration of the solvable Lie algebra $\mathfrak{r}((2,3,5,7),(2,3,4,1))$ via its associated labeled directed graph, as per Definition~\ref{def:assocaited:graph}, using the convention to scale the non-loop edges according to the magnitude of the structure constants corresponding to the encoded bracket. Here, the label $y_{d_j k}$ denotes the basis element $y_{jk}$ from Definition~\ref{def:rnm}.}
    \label{fig:rdm}
\end{figure}

Let us also introduce key properties for the subsequent analysis.


\begin{definition}[abelian centralizer condition]
    We say a Lie algebra $\g$ satisfies the \emph{abelian centralizer condition} (acc) if and only if the centralizer $C_\g(x)$ is abelian for every $x\notin\mathcal{Z}(\g)$.
\end{definition}


\begin{definition}[commutative transitive]
    We say a Lie algebra $\g$ is \emph{commutative transitive} (ct) if and only if the centralizer $C_\g(x)$ is abelian for every $x\in\g\setminus\{0\}$.
\end{definition}

Clearly, a ct Lie algebra satisfies the acc. Conversely, the acc implies ct when $\mathcal{Z}(\g)=\{0\}$. A non-abelian ct Lie algebra has zero center, since the centralizer of any non-zero central element would be the whole Lie algebra. Thus an acc statement for $\mathcal{U}$ is not a ct statement for $\mathcal{U}$ or for an arbitrary quotient; the required quotient argument is given below.

\subsection{Preliminary structural results}

First, we want to establish the isomorphism classes of the Lie algebras introduced in Definition~\ref{def:rnm}.


\begin{lemma}\label{lem:rdm:isomorphism:classes}
    Let $m,m'\in \N_{>0}$, $\boldsymbol{m},\boldsymbol{d}\in\N_{>0}^m$ and $\boldsymbol{m}',\boldsymbol{d}'\in \N_{>0}^{m'}$ such that $d_1,\ldots, d_m$ and  $d_1',\ldots,d_{m'}'$ are both pairwise distinct. Then it holds that $\mathfrak{r}(\boldsymbol{d},\boldsymbol{m})\cong \mathfrak{r}(\boldsymbol{d}',\boldsymbol{m}')$ if and only if $m=m'$ and there exists a positive rational constant $\kappa\in\Q_{>0}$ and a permutation $\sigma\in S_m$ such that $d_j=\kappa d_{\sigma(j)}'$ and $m_j=m_{\sigma(j)}'$ for all $j\in\{1,\ldots,m\}$. 
\end{lemma}

\begin{proof}
    For convenience, let us write $\g\equiv \mathfrak{r}(\boldsymbol{d},\boldsymbol{m})$ and $\g'\equiv \mathfrak{r}(\boldsymbol{d}',\boldsymbol{m}')$, and denote their respective bases by $\mathcal{B}:=\{x,y_{jk_j}\,\mid\,j\in \{1,\ldots,m\},k_j\in\{1,\ldots,m_j\}\}$ and $\mathcal{B}':=\{x',y_{jk_j}'\,\mid\,j\in \{1,\ldots,m'\},k_j\in\{1,\ldots,m_j'\}\}$, respectively. Now suppose $\varphi:\g\to\g'$ is a Lie-algebra isomorphism. This implies that both Lie algebras must have the same dimension, i.e., $1+\sum_{j=1}^mm_j=1+\sum_{j=1}^{m'}m_j'$. Moreover, by construction, it holds for the derived algebras that $\mathcal{D}\g=\spn\{\mathcal{B}\setminus\{x\}\}$ and $\mathcal{D}\g'=\spn\{\mathcal{B}'\setminus\{x'\}\}$. Indeed, every defining bracket lies in $\spn\{\mathcal{B}\setminus\{x\}\}$. Hence, by bilinearity, $\mathcal{D}\g\subseteq\spn\{\mathcal{B}\setminus\{x\}\}$. Conversely, since each $d_j$ is nonzero, it holds that every $y_{jk_j}=d_j^{-1}[x,y_{jk_j}]\in \mathcal{D}\g$ for all admissible $j$ and $k_j$. Thus, $\mathcal{D}\g=\spn\{\mathcal{B}\setminus\{x\}\}$. The same argument holds for $\g'$. Note that, since $[y_{jk_j},y_{\ell k_{\ell}}]=0$, for all admissible indices $j,\ell,k_j,k_\ell$, this also implies that both $\mathcal{D}\g$ and $\mathcal{D}\g'$ are abelian.

    Since $\varphi$ is a Lie-algebra-isomorphism, and therefore preserves the bracket structure, while also being surjective, it holds that $\varphi(\mathcal{D}\g)=\mathcal{D}\g'$. This implies that $\varphi(y_{jk_j})\in \mathcal{D}\g'$ for all admissible $j,k_j$. We may now write $\varphi(x)=\kappa x'+r'$, where $\kappa\in \C$ and $r'\in \mathcal{D}\g'$. If $\kappa=0$ it would hold that $\varphi(y_{jk_j})=d_j^{-1}\varphi([x,y_{jk_j}])=d_j^{-1}[r',\varphi(y_{jk_j})]=0$ for all admissible $j,k_j$, and therefore $\varphi(\mathcal{D}\g)=\{0\}$, which contradicts the surjectivity, since $\dim(\mathcal{D}\g')=\sum_{j=1}^{m'}m_j'>0$. Thus, $\kappa\neq 0$. Next, we observe that, by construction
    \begin{align*}
        \g&=\spn\{x\}\oplus\bigoplus_{j=1}^m\g_{d_j},\qquad\text{with}\qquad\g_{d_j}=\spn\{y_{j k_j}\,\mid\, k_j\in\{1,\ldots,m_j\}\}=\{ v\in \g\,\mid\,[x,v]=d_j v\}.
    \end{align*}
    Similarly, it also holds that
    \begin{align*}
        \g'&=\spn\{\kappa x'\}\oplus\bigoplus_{j=1}^{m'}\g_{\kappa d_j'}',\qquad\text{with}\qquad\g_{\kappa d_j'}'=\spn\{y_{j k_j}'\,\mid\, k_j\in\{1,\ldots,m_j'\}\}=\{ v'\in \g'\,\mid\,[\kappa x',v']=\kappa d_j' v'\}.
    \end{align*}
    Let now $v\in \g_{d_j}\setminus\{0\}$ for some $j\in\{1,\ldots,m\}$, then
    \begin{align*}
        \varphi(v)=\frac{1}{d_j}\varphi([x,v])=\frac{1}{d_j}[\varphi(x),\varphi(v)]=\frac{1}{d_j}[\kappa x',\varphi(v)],
    \end{align*}
    since $d_j\neq 0$ and $r',\varphi(v)\in \mathcal{D}\g'$ and therefore $[r',\varphi(v)]=0$. It is furthermore essential to note that $\varphi(v)\neq 0$, since $\varphi$ is a Lie-algebra isomorphism and therefore $\varphi(v)=0$ if and only if $v=0$. This shows that $\varphi(v)$ is a non-zero eigenvector of $\operatorname{ad}_{\kappa x'}$ to eigenvalue $d_j$. The non-zero eigenvalues of $\operatorname{ad}_{\kappa x'}$ are $\kappa d_{1}',\ldots,\kappa d_{m'}'$. Thus, it must hold that $d_j=\kappa d_{j'}'$ for some $j'\in\{1,\ldots,m'\}$. Moreover, it holds that $\varphi(\g_{d_j})\subseteq \g_{d_j}'=\g_{\kappa d_{j'}'}'$, since also $\varphi(0)=0$. Conversely, let $u'\in \g_{d_j}'=\g_{\kappa d_{j'}'}'$. Then, due to the Lie-algebra isomorphism $\varphi|_{\mathcal{D}\g}:\mathcal{D}\g\to\mathcal{D}\g'$, there exists some $u\in\mathcal{D}\g$ such that $\varphi(u)=u'$. It now holds that
    \begin{align*}
        \varphi(u)=u'=\frac{1}{d_j}[\kappa x',u']=\frac{1}{d_j}[\kappa x'+r',u']=\frac{1}{d_j}[\varphi(x),\varphi(u)].
    \end{align*}
    Injectivity now implies that $[x,u]=d_j u$, hence $u\in \g_{d_j}$. It consequently follows that $\varphi(\g_{d_j})\supseteq \g_{d_j}'=\g_{\kappa d_{j'}'}'$ and therefore $\varphi(\g_{d_j})= \g_{d_j}'=\g_{\kappa d_{j'}'}'$.

    This holds for all $j\in\{1,\ldots,m\}$, which also shows that distinct $d_j$ are matched with distinct $\kappa d_{j'}'$. Because $\varphi(\mathcal{D}\g)=\mathcal{D}\g'$ and all target blocks have the same non-zero dimension as the source blocks, every target block must occur in this matching, otherwise that block would be absent from the sum of the images. Thus $m=m'$ and there exists a one-to-one map $\sigma:\{1,\ldots,m\}\to\{1,\ldots,m'\}$, such that $j'=\sigma(j)$. Moreover, matching the dimensions of the individual blocks yields $m_j=\dim(\g_{d_j})=\dim(\g_{\kappa d_{\sigma(j)}'}')=m_{\sigma(j)}'$. Finally, since $d_j=\kappa d_{\sigma(j)}'$, where both $d_j$ and $d_{\sigma(j)}'$ are positive integers, it holds that $\kappa\in \Q_{>0}$. This shows that if $\g\cong\g'$, then $m=m'$ and there exists a positive rational constant $\kappa\in\Q_{>0}$ and a permutation $\sigma\in S_m$ such that $d_j=\kappa d_{\sigma(j)}'$ and $m_j=m_{\sigma(j)}'$ for all $j\in\{1,\ldots,m\}$.

    Let us now consider the converse direction. Then, we define the linear map $\psi:\g\to\g'$ by setting $\psi(x)=\kappa x'$ and $\psi(y_{jk})=y_{\sigma(j)k}'$ for $k\in\{1,\ldots,m_j=m_{\sigma(j)}'\}$. This map is clearly invertible, since $\kappa\neq 0$, $\sigma$ is a permutation, and the matched multiplicities agree, which makes $\psi$ a bijection between the two Lie algebras $\g$ and $\g'$. It holds furthermore, for all non-trivial Lie brackets,
    \begin{align*}
        \psi([x,y_{jk}])=d_j\psi(y_{j k})=\kappa d_{\sigma(j)}' y_{\sigma(j)k}'=[\kappa x' , y_{\sigma(j)k}']=[\psi(x),\psi(y_{jk})],
    \end{align*}
    while all Lie brackets between images of the $y$ basis elements vanish, as do their source brackets.
    Bilinearity of the Lie brackets now proves that $\psi$ is a Lie-algebra isomorphism between $\g$ and $\g'$.
\end{proof}

Thus, ordering and normalizing $\boldsymbol{d}$ in Definition~\ref{def:rnm} yields a unique parameter pair for each isomorphism class of a Lie algebra $\mathfrak{r}(\boldsymbol{d},\boldsymbol{m})$.

\begin{corollary}\label{cor:rdm:Cd0:isomorphism:uniqueness}
    Let $d_0,d_0'\in\Np{}$, $m,m'\in \N_{>0}$, $\boldsymbol{m},\boldsymbol{d}\in\N_{>0}^m$ and $\boldsymbol{m}',\boldsymbol{d}'\in \N_{>0}^{m'}$ such that $d_1<d_2<\ldots<d_m$, $d_1'<d_2'<\ldots<d_{m'}'$, and $\operatorname{gcd}(d_1,\ldots,d_m)=1=\operatorname{gcd}(d_1',\ldots,d_{m'}')$. Then it holds that $\mathfrak{r}(\boldsymbol{d},\boldsymbol{m})\oplus\C^{d_0}\cong \mathfrak{r}(\boldsymbol{d}',\boldsymbol{m}')\oplus\C^{d_0'}$ if and only if $d_0=d_0'$, $m=m'$, $\boldsymbol{d}=\boldsymbol{d}'$, and $\boldsymbol{m}=\boldsymbol{m}'$. 
\end{corollary}

\begin{proof}
    Let us first consider the case $d_0=d_0'=0$. Then the statement is an immediate consequence of Lemma~\ref{lem:rdm:isomorphism:classes}, as $\mathfrak{r}(\boldsymbol{d},\boldsymbol{m})\cong \mathfrak{r}(\boldsymbol{d}',\boldsymbol{m}')$ if and only if $m=m'$ and there exists a positive rational constant $\kappa\in\Q_{>0}$ and a permutation $\sigma\in S_m$ such that $d_j=\kappa d_{\sigma(j)}'$ and $m_j=m_{\sigma(j)}'$ for all $j\in\{1,\ldots,m\}$. Since $\kappa>0$, and $d_1<d_2<\ldots<d_m$, as well as $d_1'<d_2'<\ldots<d_{m'}'=d_m'$, it must hold that $\sigma=\operatorname{id}_{\{1,\ldots,m\}}$ is the identity permutation on the set $\{1,\ldots,m\}$. We may write $\kappa= a/b$, where $a,b\in\N_{>0}$. Then, since $d_j=\kappa d_j'$ and therefore $b d_j= a d_{j}'$, it holds that
    \begin{align*}
        b=b\operatorname{gcd}(d_1,\ldots,d_m)=\operatorname{gcd}(bd_1,\ldots,bd_m)=\operatorname{gcd}(ad_1',\ldots,a d_m')=a \operatorname{gcd}(d_1',\ldots,d_m')=a.
    \end{align*}
    Thus, $\kappa=1$, and hence $\mathfrak{r}(\boldsymbol{d},\boldsymbol{m})\cong \mathfrak{r}(\boldsymbol{d}',\boldsymbol{m}')$ if and only if $m=m'$, $\boldsymbol{d}=\boldsymbol{d}'$, and $\boldsymbol{m}=\boldsymbol{m}'$. 

    Let now $d_0,d_0'\in\Np{}$. It is straightforward to observe that, by construction, $\mathcal{Z}(\mathfrak{r}(\boldsymbol{d},\boldsymbol{m}))=\{0\}$. Thus, $\mathcal{Z}(\mathfrak{r}(\boldsymbol{d},\boldsymbol{m})\oplus\C^{d_0})=\C^{d_0}$, and $(\mathfrak{r}(\boldsymbol{d},\boldsymbol{m})\oplus\C^{d_0})/\mathcal{Z}(\mathfrak{r}(\boldsymbol{d},\boldsymbol{m})\oplus\C^{d_0})\cong\mathfrak{r}(\boldsymbol{d},\boldsymbol{m})$. The same argument applies to $\mathfrak{r}(\boldsymbol{d}',\boldsymbol{m}')\oplus\C^{d_0'}$. An isomorphism preserves the centers and induces an isomorphism between the corresponding quotients. Thus, $\mathfrak{r}(\boldsymbol{d},\boldsymbol{m})\oplus\C^{d_0}\cong \mathfrak{r}(\boldsymbol{d}',\boldsymbol{m}')\oplus\C^{d_0'}$ implies that $\mathfrak{r}(\boldsymbol{d},\boldsymbol{m})\cong \mathfrak{r}(\boldsymbol{d}',\boldsymbol{m}')$ and $d_0=d_0'$. Conversely, an isomorphism between $\mathfrak{r}(\boldsymbol{d},\boldsymbol{m})$ and $\mathfrak{r}(\boldsymbol{d}',\boldsymbol{m}')$ together with $d_0=d_0'$ yields an isomorphism of the direct sums. The claim now follows from the case $d_0=d_0'=0$.
\end{proof}

Note that existence of the parameter choices in Corollary~\ref{cor:rdm:Cd0:isomorphism:uniqueness} above is given by the following observation: Let $\mathfrak{r}(\boldsymbol{d},\boldsymbol{m})$ with $m\in\N_{>0}$, $\boldsymbol{d},\boldsymbol{m}\in \N_{>0}^m$ be as defined earlier. Set $g:=\operatorname{gcd}(d_1,\ldots,d_m)$, and choose $\sigma\in S_m$ such that $d_{\sigma(1)}<d_{\sigma(2)}<\ldots<d_{\sigma(m)}$. Then, the assignments $d_j':= d_{\sigma(j)}/g$ and $m_j':= m_{\sigma(j)}$ along with the replacement $x':=x/g$ and the relabeling of the remaining basis elements as $y_{jk}':= y_{\sigma(j)k}$ yield for the only non-trivial Lie brackets $[x',y_{jk}']=d_j' y_{jk}'$ while the brackets between the $y'$ elements remain unchanged. Hence $\mathfrak{r}(\boldsymbol{d},\boldsymbol{m})\cong \mathfrak{r}(\boldsymbol{d}',\boldsymbol{m}')$, with increasing $d_j'$ of $\operatorname{gcd}$ 1. Their uniqueness follows from the corollary above.

Next, let us obtain some structural consequences that hold for Lie algebras that satisfy the acc or are ct.


\begin{lemma}\label{lem:subalgebra:acc}
    Let $\mathfrak{l}$ be a non-abelian Lie algebra that satisfies the acc and satisfies $\mathfrak{l}=[\mathfrak{l},\mathfrak{l}]\oplus\mathcal{Z}(\mathfrak{l})$. Furthermore, let $\g$ be a non-abelian Lie subalgebra of $\mathfrak{l}$. Then $\g$ satisfies the acc and $\mathfrak{h}=\g/(\g\cap \mathcal{Z}(\mathfrak{l}))$ is ct.
\end{lemma}

\begin{proof}
    Since $\g\subseteq\mathfrak{l}$, every $x\in \g\setminus\mathcal{Z}(\g)$ also lies outside of $\mathcal{Z}(\mathfrak{l})$. Thus $C_\g(x)=\g\cap C_{\mathfrak{l}}(x)$ is abelian by the acc for $\mathfrak{l}$, so $\g$ satisfies also the acc.
    
    Observe that $\g\cap\mathcal{Z}(\mathfrak{l})$ is a central ideal of $\g$, making the quotient $\mathfrak{h}:=\g/(\g\cap \mathcal{Z}(\mathfrak{l}))$ a Lie algebra. Let now $\Bar{x}=x+(\g\cap\mathcal{Z}(\mathfrak{l}))\in\mathfrak{h}$, with $x\in \g\setminus\mathcal{Z}(\mathfrak{l})$, and let $\Bar{y}=y+(\g\cap\mathcal{Z}(\mathfrak{l})),\Bar{z}=z+(\g\cap\mathcal{Z}(\mathfrak{l}))\in C_{\mathfrak{h}}(\Bar{x})$, with $y,z\in \g$. It follows that $[\Bar{x},\Bar{y}]=[x,y]+ (\g\cap\mathcal{Z}(\mathfrak{l}))=0+(\g\cap\mathcal{Z}(\mathfrak{l}))=\Bar{0}$ and $[\Bar{x},\Bar{z}]=[x,z]+ (\g\cap\mathcal{Z}(\mathfrak{l}))=\Bar{0}$. Thus, $[x,y],[x,z]\in (\g\cap\mathcal{Z}(\mathfrak{l}))$. However, by construction $[x,y],[x,z]\in [\mathfrak{l},\mathfrak{l}]$. Thus, the assumption that $\mathfrak{l}=[\mathfrak{l},\mathfrak{l}]\oplus\mathcal{Z}(\mathfrak{l})$ implies that $[x,y]=[x,z]=0$, i.e., that $y,z\in C_\g(x)$. The acc for $\mathfrak{l}$ now yields that $[y,z]=0$, since $x\notin \mathcal{Z}(\mathfrak{l})$. Hence, $[\Bar{y},\Bar{z}]=[y,z]+(\g\cap\mathcal{Z}(\mathfrak{l}))=0+(\g\cap\mathcal{Z}(\mathfrak{l}))=\Bar{0}$. Since $\Bar{y},\Bar{z}$ have been chosen arbitrarily, it must hold that $C_\mathfrak{h}(\Bar{x})$ is abelian, showing that $\mathfrak{h}$ is ct.
\end{proof}


\begin{lemma}\label{lem:acc:semisimple:sl2}
    Let $\g$ be a non-zero finite-dimensional semisimple Lie algebra that satisfies the acc. Then $\g$ is simple and it holds furthermore that $\g\cong\sl{2}{\C}$.
\end{lemma}

\begin{proof}
    Any non-zero semisimple Lie algebra $\g$ can be decomposed into $m\geq 1$ mutually commutative simple ideals $\g_j$ such that
    \begin{align*}
        \g=\bigoplus_{j=1}^m\g_j,\qquad[\g_j,\g_k]=\left\{\begin{matrix}
            \g_j&,\text{ if }j=k,\\
            \{0\}&,\text{ if }j\neq k
        \end{matrix}\right.
    \end{align*}
    as demonstrated in~\cite{Knapp:1996}. Now suppose that $m\geq 2$. Then there would exist $j,k\in\{1,\ldots,m\}$ such that $j\neq k$, as well as $x_j,y_j\in \g_j\setminus\{0\}$ and $x_k,y_k\in \g_{k}$ such that $[x_j,y_j]\neq0\neq[x_k,y_k]$, since $\g_j,\g_k$ are simple and therefore non-abelian Lie algebras. This implies, in particular, that $x_j\notin \mathcal{Z}(\g)$. By construction, it holds that $[x_j,x_k]=0=[x_j,y_k]$. Thus $x_k,y_k\in C_\g(x_j)$. The acc would now imply that $[x_k,y_k]=0$, contradicting the initial assumption that these two elements do not commute. Thus, $m=1$, and $\g$ is simple. 

    Moreover, for any simple Lie algebra $\g$, it holds that $\mathcal{Z}(\g)=\{0\}$. Thus any simple Lie algebra that satisfies the acc is ct. Theorem~6 from~\cite{Klep:2010} now yields that $\g\cong\sl{2}{\C}$.
\end{proof}

\begin{lemma}\label{lem:structural:lietarture:results}
    It holds that (a) $\mathcal{U}=[\mathcal{U},\mathcal{U}]\oplus\mathcal{Z}(\mathcal{U})$, (b) $\mathcal{U}$ satisfies the acc, and (c) $\mathcal{U}$ is a domain.
\end{lemma}

\begin{proof}
    Claim~(a) has, for instance, been established in the proof of Theorem~B.18 in~\cite{Heib:2026:SSRC:Universality}.  Claim~(b) follows from Proposition~6.1 in~\cite{Bavula:2006}, as the associative commutative-centralizer condition investigated there implies the acc for the commutator Lie algebra, because the two centralizers have the same underlying set. Claim~(c) is also recalled at the beginning of Section~6 in~\cite{Bavula:2006}.
\end{proof}

\section{Proof of the main classification result}\label{sec:main:proof}

In this section, we provide the proof of Theorem~\ref{thm:classification:Usl2C}.

\subsection{Nilpotent Lie subalgebras}


\begin{lemma}\label{lem:nilpotent:U:subalgebras}
    Let $\g$ be a finite-dimensional nilpotent Lie subalgebra of $\mathcal{U}$. Then $\g$ is abelian.
\end{lemma}

\begin{proof}
    Suppose $\mathfrak{n}\subseteq\mathcal{U}$ is a non-zero finite-dimensional non-abelian nilpotent Lie subalgebra of $\mathcal{U}$. Then, there exists a largest positive integer $k\in\N_{\geq1}$ such that $\mathcal{C}^k\mathfrak{n}\neq\{0\}$, while $\mathcal{C}^\ell\mathfrak{n}=\{0\}$ for all $\ell\geq k+1$. By Definition~\ref{def:lie:algebra:concepts}, it follows from $\mathcal{C}^{k+1}\mathfrak{n}=[\mathfrak{n},\mathcal{C}^k\mathfrak{n}]=\{0\}$ that $\mathcal{C}^k\mathfrak{n}\subseteq\mathcal{Z}(\mathfrak{n})$. Now, choose $z\in \mathcal{C}^k\mathfrak{n}\setminus\{0\}$. This implies that for each $x\in\mathfrak{n}$, it holds that $x\in C_\mathcal{U}(z)$. Recall that $\mathcal{U}$ satisfies the acc by Lemma~\ref{lem:structural:lietarture:results}. There are two options: either (i) $z\in \mathcal{Z}(\mathcal{U})$, or (ii) $z\notin \mathcal{Z}(\mathcal{U})$. The latter case (ii) is impossible as it would imply, by the acc, that $[x,y]=0$ for all $x,y\in\mathfrak{n}\subseteq C_{\mathcal{U}}(z)$, which, in turn, would imply that $[\mathfrak{n},\mathfrak{n}]=\{0\}$, contradicting the initial assumption that $\mathfrak{n}$ is non-abelian. Thus, $z\in \mathcal{Z}(\mathcal{U})$. However, by construction, $z\in \mathcal{C}^k\mathfrak{n}\subseteq \mathcal{C}^1\mathfrak{n}=[\mathfrak{n},\mathfrak{n}]\subseteq[\mathcal{U},\mathcal{U}]$, which implies, by Lemma~\ref{lem:structural:lietarture:results}, the existence of a non-zero $z\in\mathcal{Z}(\mathcal{U})\cap[\mathcal{U},\mathcal{U}]=\{0\}$, which is impossible. Thus, $\mathfrak{n}$ cannot be non-abelian and nilpotent. Thus all nilpotent finite-dimensional subalgebras of $\mathcal{U}$ are abelian. 
\end{proof}

\begin{corollary}\label{cor:derived:solvable:algebra:abelian}
    Let $\g$ be a finite-dimensional solvable Lie subalgebra of $\mathcal{U}$. Then the derived algebra $\mathcal{D}\g$ is abelian.
\end{corollary}

\begin{proof}
    Since $\g$ is solvable, its derived algebra $\mathcal{D}\g$ is nilpotent~\cite{Knapp:1996}\footnote{For this conclusion to be valid it is imperative that the Lie algebra of interest is taken over a field of characteristic zero~\cite{Jacobson:1979}. Thus, this assertion is correct here, since the complex numbers constitute a characteristic-zero field.}. Since $\mathcal{D}\g$ is also a finite-dimensional Lie subalgebra of $\mathcal{U}$, we may employ Lemma~\ref{lem:nilpotent:U:subalgebras} to conclude that $\mathcal{D}\g$ is abelian.
\end{proof}

\subsection{Semisimple Lie subalgebras}


\begin{lemma}\label{lem:g:semisimple:sl2c}
    Let $\g$ be a non-zero finite-dimensional semisimple Lie subalgebra of $\mathcal{U}$. Then, $\g\cong\sl{2}{\C}$.
\end{lemma}

\begin{proof}
    This is a direct consequence of Lemmata~\ref{lem:subalgebra:acc} and~\ref{lem:acc:semisimple:sl2}, since Lemma~\ref{lem:structural:lietarture:results} recalls that $\mathcal{U}$ satisfies the acc.
\end{proof}

\subsection{Non-solvable Lie subalgebras}


\begin{lemma}\label{lem:semisimple:U:subalgebra:sl2}
    Let $\g$ be a non-zero finite-dimensional non-solvable Lie subalgebra of $\mathcal{U}$. Then $\g\cong\sl{2}{\C}\oplus\C^d$, where $d\in\Np{}$ and $\C^0\equiv \{0\}$ by convention.
\end{lemma}

\begin{proof}
    The Levi-Mal'tsev theorem implies that $\g$ may be written as the semidirect sum $\g\cong\mathfrak{s}\ltimes\mathfrak{r}$, where $\mathfrak{s}$ is a semisimple Lie subalgebra of $\g$ and $\mathfrak{r}$ is its largest solvable ideal, the radical~\cite{Levi:1905,Matltsev:1942,Serre:1992}. Both are Lie subalgebras of $\mathcal{U}$, and since $\g$ is non-solvable $\mathfrak{s}\neq\{0\}$. Thus, by Lemma~\ref{lem:g:semisimple:sl2c}, it must hold that $\mathfrak{s}\cong \sl{2}{\C}$. 

    We now define the quotient map $q:\g\to\mathfrak{h}:=\g/(\g\cap \mathcal{Z}(\mathcal{U}))$ and write $\Bar{u}:=q(u)=u+(\g\cap\mathcal{Z}(\mathcal{U}))$ for $u\in\g$. Note that the denominator is a central ideal of $\g$, making its quotient $\mathfrak{h}$ a Lie algebra. We now want to show that $\mathfrak{h}$ is non-solvable. For this, we recognize that, for all $x,y\in \g$,
    \begin{align*}
        q([x,y])=[x,y]+(\g\cap\mathcal{Z}(\mathcal{U}))=[x+(\g\cap\mathcal{Z}(\mathcal{U})),y+(\g\cap\mathcal{Z}(\mathcal{U}))]=[q(x),q(y)].
    \end{align*}
    Thus, since a quotient map is linear, $q$ is a Lie-algebra homomorphism~\cite{Humphreys:1972}. It follows, since $q$ is by definition surjective, that
    \begin{align*}
        q(\mathcal{D}\g)&=q\left(\spn\left\{[x,y]\,\mid\,x,y\in \g\right\}\right)=\spn\left\{q([x,y])\,\mid\,x,y\in \g\right\}=\spn\{[q(x),q(y)]\,\mid\, x,y\in \g\}\\
        &=\spn\{[\Bar{x},\Bar{y}]\,\mid\, \Bar{x},\Bar{y}\in\mathfrak{h}\}=\mathcal{D}\mathfrak{h}=\mathcal{D}q(\g).
    \end{align*}
    That is, $q$ and $\mathcal{D}$ commute, which also implies that $q(\mathcal{D}^\ell\g)=\mathcal{D}^\ell q(\g)=\mathcal{D}^\ell \mathfrak{h}$ for all $\ell\in\Np{}$. Indeed, this follows by induction. The base case has been shown and the induction step is immediate, because
    \begin{align*}
        q(\mathcal{D}^{\ell+1}\g)=q([\mathcal{D}^\ell\g,\mathcal{D}^\ell\g])=[q(\mathcal{D}^\ell\g),q(\mathcal{D}^\ell\g)]=[\mathcal{D}^\ell\mathfrak{h},\mathcal{D}^\ell\mathfrak{h}]=\mathcal{D}^{\ell+1}\mathfrak{h}.
    \end{align*}
    Since $\g$ is, by assumption, finite-dimensional and non-solvable, there exists some $\ell\in\N_{>0}$ such that $\mathcal{D}^\ell\g=\mathcal{D}^{\ell+1}\g\neq\{0\}$. The stated centre-commutator splitting, recalled in Lemma~\ref{lem:structural:lietarture:results}, implies, in particular, that $\mathcal{D}^\ell\g\cap(\g\cap\mathcal{Z}(\mathcal{U}))\subseteq[\mathcal{U},\mathcal{U}]\cap\mathcal{Z}(\mathcal{U})=\{0\}$. Hence the restriction of $q$ to $\mathcal{D}^\ell\g$ is injective and its image is non-zero. It follows that $\{\Bar{0}\}\neq\mathcal{D}^\ell \mathfrak{h}=q(\mathcal{D}^\ell \g)=q(\mathcal{D}^{\ell+1}\g)=\mathcal{D}^{\ell+1}\mathfrak{h}$, which demonstrates that $\mathfrak{h}$ is non-solvable. 
    
    Lemma~\ref{lem:subalgebra:acc} now implies that $\mathfrak{h}$ is ct, and therefore satisfies the acc. A non-solvable ct Lie algebra is by Theorem~7 from~\cite{Klep:2010} semisimple.  Due to Lemma~\ref{lem:acc:semisimple:sl2}, it follows that $\mathfrak{h}\cong\sl{2}{\C}$. Thus, the solvable ideal $q(\mathfrak{r})$ is $\{\Bar{0}\}$, which implies that $\mathfrak{r}\subseteq \g\cap\mathcal{Z}(\mathcal{U})\subseteq\mathcal{Z}(\g)$. Since $\mathfrak{s}\cap\mathcal{Z}(\g)=\{0\}$, and $\g=\mathfrak{s}\ltimes\mathfrak{r}$, it follows that $\mathfrak{r}=\g\cap\mathcal{Z}(\g)$. Consequently, the representation defining the semidirect sum $\mathfrak{s}\ltimes\mathfrak{r}$ is the trivial one, showing that $\g=\mathfrak{s}\oplus\mathcal{Z}(\g)$ as Lie algebras. Hence, $\g\cong\sl{2}{\C}\oplus\C^d$, where $d:=\dim(\mathcal{Z}(\g))\in\Np{}$.
\end{proof}

\subsection{Solvable non-nilpotent Lie subalgebras}

For clarity, we split the proof of the solvable non-nilpotent classification Lemma~\ref{lem:solvable:non:nilpotent} into several auxiliary lemmata. Throughout these auxiliary arguments, $\g\subseteq\mathcal{U}$ is a finite-dimensional solvable non-nilpotent Lie subalgebra, and $x\in\g$ is chosen such that $\operatorname{ad}_x:\g\to\g, y\mapsto\operatorname{ad}_x(y):=[x,y]$ is not nilpotent. Moreover, we define the index sets
\begin{align*}
    \mathcal{I}_{k}^{(N)}:=\left\{(n_1,n_2,\ldots,n_k)\in \Np{k}\,\mid\, \sum_{j=1}^k n_j\leq N\right\}\qquad\text{and}\qquad\Bar{\mathcal{I}}_{k}^{(n)}:=\left\{(n_1,n_2,\ldots,n_k)\in \Np{k}\,\mid\, \sum_{j=1}^k n_j=n\right\}
\end{align*}
for ease of presentation. The proof proceeds as follows: first, we show that, for any $x\in\mathcal{U}$ for which $\operatorname{ad}_x$ possesses a non-zero eigenvalue, $C$ and $x$ cannot vanish under any finite polynomial relation except the trivial one. Then, we show that any non-nilpotent $\operatorname{ad}_x$ is diagonalizable and that the ratios of its non-zero eigenvalues are positive and rational. Finally, we normalize these eigenvalues by rescaling $x$ and separate the center to obtain the stated direct-sum decomposition.

\begin{lemma}\label{lem:polynom:C:x:zero:if:x:nonnilpotent}
    Let $x\in \mathcal{U}$, such that $\operatorname{ad}_x$ has an eigenvector $y\in\mathcal{U}\setminus\{0\}$ to a non-zero eigenvalue $\lambda\in\C\setminus\{0\}$, i.e., $\operatorname{ad}_x(y)=\lambda y\neq 0$. Moreover, let $p\in\C[a,b]$ be a finite complex polynomial. Then $p(C,x)=0$ if and only if $p=0$.
\end{lemma}

\begin{proof}
    Suppose there exists a finite complex polynomial $p\in\C[a,b]\setminus\{0\}$ satisfying $p(C,x)=0$ of least possible degree in the second argument. We can expand it as
    \begin{align*}
        p(C,x)=\sum_{\alpha=0}^{d_1}\sum_{\beta=0}^{d_2}\theta_{\alpha\beta}C^\alpha x^\beta=0,
    \end{align*}
    where $d_1:=\deg_a(p)=\max\{\alpha\,\mid\,\theta_{\alpha\beta}\neq 0\text{ for some }\beta\in\Np{}\}$ and $d_2:=\deg_b(p)=\max\{\beta\,\mid\,\theta_{\alpha \beta}\neq 0\text{ for some }\alpha\in\Np{}\}$.
    Let us first consider the case that $d_2\geq 1$. 
    Recall that the Lie bracket of $\mathcal{U}$ is the commutator. Thus, $\operatorname{ad}_x(y)=xy-yx=\lambda y$ implies $xy=y(x+\lambda\unity)$. For $\beta\geq 1$, iteratively applying this identity yields therefore $x^\beta y=x^{\beta-1} y(x+\lambda\unity)=\ldots=y(x+\lambda\unity)^\beta$. The same identity holds for $\beta=0$. It follows that
    \begin{align*}
        0&=p(C,x)y=\sum_{\alpha=0}^{d_1}\sum_{\beta=0}^{d_2}\theta_{\alpha\beta}C^\alpha x^\beta y=\sum_{\alpha=0}^{d_1}\sum_{\beta=0}^{d_2}\theta_{\alpha\beta}C^\alpha y(x+\lambda\unity)^\beta=yp(C,x+\lambda\unity),
    \end{align*}
    since $C$ and $y$ commute. Recall that $\mathcal{U}$ is, according to Lemma~\ref{lem:structural:lietarture:results}, a domain, which implies that $p(C,x+\lambda\unity)=0$, and consequently $p(C,x)=p(C,x+\lambda\unity)$ holds. Hence, the polynomial $q(a,b):=p(a,b+\lambda)-p(a,b)$ also satisfies $q(C,x)=0$. The coefficient of $b^{d_2-1}$ in $q$ is
    \begin{align*}
        d_2\lambda \sum_{\alpha=0}^{d_1}\theta_{\alpha d_2}a^\alpha\neq 0.
    \end{align*}
    Thus $q$ is non-zero, satisfies $q(C,x)=0$, and has a smaller degree in the second argument than $p$, contradicting the choice of $p$. Thus, it must hold that $d_2=0$, and consequently $p(C,x)=\sum_{\alpha=0}^{d_1}\theta_\alpha C^\alpha$, where, without loss of generality, $\theta_{d_1}\neq 0$. Here, we may now expand $p(C,x)$ into PBW normal ordering by using $C=X^2+Y^2+Z^2$. The coefficient of $Z^{2d_1}$ is $\theta_{d_1}$, so linear independence of PBW monomials with different multi-degrees and vanishing of $p(C,x)$ would imply that $\theta_{d_1}=0$, contradicting the earlier assumption that $\theta_{d_1}\neq 0$. Hence $p(C,x)=0$ implies $p=0$; the converse is immediate.
\end{proof}

\begin{lemma}\label{lem:generalized:eigenvectors}
    Let $x,y,z\in\mathcal{U}\setminus\{0\}$ and $\lambda\in\C\setminus\{0\}$ satisfy $[x,y]=\lambda y$, $[x,z]=\lambda z+y$, $[y,z]=0$. Then, for all $j,k\in\Np{}$, the term $y^j z^k$ is a generalized eigenvector of $\operatorname{ad}_x$ of rank $k+1$ to eigenvalue $\lambda(j+k)$. Moreover, for a polynomial $p\in\C[a,b]\setminus\{0\}$, the term $p(C,x) y^j z^k$ is also a generalized eigenvector of $\operatorname{ad}_x$ of rank $k+1$ to eigenvalue $\lambda(j+k)$. 
\end{lemma}

\begin{proof}
    Suppose there exist $x,y,z\in \mathcal{U}\setminus\{0\}$ such that $\operatorname{ad}_x(y)=\lambda y$ and $\operatorname{ad}_x(z)=\lambda z+y$ for $\lambda\in\C\setminus\{0\}$, as well as $[y,z]=0$. Then, $\operatorname{ad}_x$ is invertible on the two-dimensional subspace $\spn\{y,z\}$, and thus $\spn\{y,z\}\subseteq[\mathcal{U},\mathcal{U}]$. In the following calculation, terms that would contain a negative power are omitted before they are evaluated. This is justified, since their prefactors will vanish. Since $\operatorname{ad}_x$ is a derivation on $\mathcal{U}$, one has therefore
    \begin{align*}
        \operatorname{ad}_x(y^jz^k)&=j\operatorname{ad}_x(y)y^{j-1}z^k+k\operatorname{ad}_x(z)y^jz^{k-1}=j\lambda y^jz^k+k(\lambda z+y)y^jz^{k-1}=(j+k)\lambda y^jz^k+ky^{j+1}z^{k-1},
    \end{align*}
    for all $j,k\in\Np{}$. We can extend this formula to arbitrary powers of $\operatorname{ad}_x$. That is, for all $\ell\in\Np{}$, one has
    \begin{align*}
        \operatorname{ad}_x^\ell(y^jz^k)&=\sum_{m=0}^{\min\{\ell,k\}} \binom{\ell}{m}(j+k)^{\ell-m}\lambda^{\ell-m}\frac{k!}{(k-m)!}y^{j+m}z^{k-m}.
    \end{align*}
    This can be shown by induction on $\ell$. The case $\ell=0$ is trivial, since $\operatorname{ad}_x^0\equiv\operatorname{id}_\mathcal{U}$. The case $\ell=1$ is already established. Thus, it suffices to observe that
    \begin{align*}
        \operatorname{ad}_x^{\ell+1}(y^jz^k)&=\sum_{m=0}^\ell \binom{\ell}{m}(j+k)^{\ell-m}\lambda^{\ell-m}\frac{k!}{(k-m)!}\operatorname{ad}_x(y^{j+m}z^{k-m})\\
        &=\sum_{m=0}^\ell \binom{\ell}{m}(j+k)^{\ell-m+1}\lambda^{\ell-m+1}\frac{k!}{(k-m)!}y^{j+m}z^{k-m}+\sum_{m=0}^\ell \binom{\ell}{m}(j+k)^{\ell-m}\lambda^{\ell-m}\frac{k!}{(k-m-1)!}y^{j+m+1}z^{k-m-1}\\
        &=\sum_{m=0}^\ell \binom{\ell}{m}(j+k)^{\ell+1-m}\lambda^{\ell+1-m}\frac{k!}{(k-m)!}y^{j+m}z^{k-m}+\sum_{m=1}^{\ell+1}\binom{\ell}{m-1}(j+k)^{\ell+1-m}\lambda^{\ell+1-m}\frac{k!}{(k-m)!}y^{j+m}z^{k-m}\\
        &=\binom{\ell}{0}(j+k)^{\ell+1}\lambda^{\ell+1}y^jz^k+\sum_{m=1}^\ell\left[\binom{\ell}{m}+\binom{\ell}{m-1}\right](j+k)^{\ell+1-m}\lambda^{\ell+1-m}\frac{k!}{(k-m)!}y^{j+m}z^{k-m}\\
        &\quad+\binom{\ell}{\ell}\frac{k!}{(k-(\ell+1))!}y^{j+\ell+1}z^{k-\ell-1}\\
        &=\binom{\ell+1}{0}(j+k)^{\ell+1}\lambda^{\ell+1}y^jz^k+\sum_{m=1}^\ell\binom{\ell+1}{m}(j+k)^{\ell+1-m}\lambda^{\ell+1-m}\frac{k!}{(k-m)!}y^{j+m}z^{k-m}\\
        &\quad+\binom{\ell+1}{\ell+1}\frac{k!}{(k-(\ell+1))!}y^{j+\ell+1}z^{k-\ell-1}\\
        &=\sum_{m=0}^{\ell+1} \binom{\ell+1}{m}(j+k)^{\ell+1-m}\lambda^{\ell+1-m}\frac{k!}{(k-m)!}y^{j+m}z^{k-m},
    \end{align*}
    where every term with a negative exponent or negative factorial argument is omitted, as stipulated before. This holds, in particular, whenever $m>k$ and therefore only in cases with $\ell+1>k$.
    
    Moreover, we also observe that
    \begin{align*}
        \operatorname{ad}_x(y^jz^k)&=(j+k)\lambda y^jz^k+ky^{j+1}z^{k-1}\qquad
        \Leftrightarrow\qquad(\operatorname{ad}_x-(j+k)\lambda\operatorname{id}_\mathcal{U})(y^jz^k)=ky^{j+1}z^{k-1}.
    \end{align*}
    It is straightforward to verify by induction that, for $0\leq \ell\leq k$, consequently
    \begin{align*}
        (\operatorname{ad}_x-(j+k)\lambda\operatorname{id}_\mathcal{U})^\ell(y^jz^k)&=\frac{k!}{(k-\ell)!}y^{j+\ell}z^{k-\ell}
    \end{align*}
    holds. Specifically, it holds that $(\operatorname{ad}_x-(j+k)\lambda\operatorname{id}_\mathcal{U})^k(y^jz^k)=k!y^{j+k}\neq 0$, where nonvanishing follows from the domain properties of $\mathcal{U}$. This implies, in particular,
    \begin{align*}
        (\operatorname{ad}_x-(j+k)\lambda\operatorname{id}_\mathcal{U})^{k+1}(y^jz^k)&=\operatorname{ad}_x(k!y^{j+k})-(j+k)\lambda k!y^{j+k}=(j+k)k!\lambda y^{j+k}-(j+k)k!\lambda y^{j+k}=0.
    \end{align*}
    Thus $y^j z^k$ is a generalized eigenvector of rank $k+1$ of the operator $\operatorname{ad}_x$ to eigenvalue $\lambda(j+k)$~\cite{Bronson:1991}. Furthermore, since both $x$ and $C$ commute with $x$, it holds that $\operatorname{ad}_x(p(C,x)y^jz^k)=(j+k)\lambda p(C,x)y^jz^k+kp(C,x)y^{j+1}z^{k-1}$, for all polynomials $p\in\C[a,b]$. Thus, analogously to before, for $p(C,x)\neq 0$, 
    \begin{align*}
        (\operatorname{ad}_x-(j+k)\lambda\operatorname{id}_\mathcal{U})^{k}p(C,x)y^jz^k=k!p(C,x) y^{j+k}\neq 0\qquad\text{and}\qquad(\operatorname{ad}_x-(j+k)\lambda\operatorname{id}_\mathcal{U})^{k+1}p(C,x)y^jz^k=0,
    \end{align*}
    which shows that the term $p(C,x)y^jz^k$ is a generalized eigenvector of $\operatorname{ad}_x$ of rank $k+1$ to eigenvalue $\lambda(j+k)$. The claim now follows from Lemma~\ref{lem:polynom:C:x:zero:if:x:nonnilpotent}, since $p(C,x)=0$ if and only if $p=0$.
\end{proof}

\begin{lemma}\label{lem:Cxyz:linearly:indepndent}
    Let either (a) $x,y,z\in\mathcal{U}\setminus\{0\}$ and $\lambda\in\C\setminus\{0\}$ satisfy $[x,y]=\lambda y$, $[x,z]=\lambda z+y$, $[y,z]=0$, or (b) $x,y,z\in\mathcal{U}\setminus\{0\}$ and $\lambda,\mu\in\C\setminus\{0\}$ satisfy $[x,y]=\lambda y$, $[x,z]=\mu z$, $[y,z]=0$, and $\lambda/\mu\notin \Q$. Then, for all $N\in\Np{}$, the set
    \begin{align*}
        \mathcal{B}_N:=\left\{C^{n_1}x^{n_2}y^{n_3}z^{n_4}\,\mid\, (n_1,n_2,n_3,n_4)\in\mathcal{I}_4^{(N)}\right\}
    \end{align*}
    is linearly independent.
\end{lemma}

\begin{proof}
    For convenience, let us write $\boldsymbol{n}:=(n_1,n_2,n_3,n_4)\in\Np{4}$. Consider the identity
    \begin{align}\label{eqn:subcliam:cxyz:linear:independent}
        \sum_{\boldsymbol{n}\in\mathcal{I}_4^{(N)}}\alpha_{\boldsymbol{n}}C^{n_1}x^{n_2}y^{n_3}z^{n_4}=0.
    \end{align}
    Next, define, for fixed $(n_3,n_4)\in\mathcal{I}_2^{(N)}$, the polynomial
    \begin{align*}
        p_{n_3,n_4}(a,b):=\sum_{(n_1,n_2)\in\mathcal{I}_2^{(N-n_3-n_4)}}\alpha_{\boldsymbol{n}}a^{n_1}b^{n_2}\in\C[a,b].
    \end{align*}
    Thus Eq.~\ref{eqn:subcliam:cxyz:linear:independent} reduces to
    \begin{align}\label{eqn:subcliam:cxyz:linear:independent:V2}
        \sum_{(n_3,n_4)\in\mathcal{I}_2^{(N)}}p_{n_3,n_4}(C,x)y^{n_3}z^{n_4}&=0.
    \end{align}
    Clearly $[x,p_{n_3,n_4}(C,x)]=0$ because $C$ is the Casimir element and $[x,x]=0$.
    
    We now group the sum in Eq.~\ref{eqn:subcliam:cxyz:linear:independent:V2} by $n=n_3+n_4$, such that it reads
    \begin{align*}
        \sum_{n=0}^N\sum_{(n_3,n_4)\in\Bar{\mathcal{I}}_2^{(n)}}p_{n_3,n_4}(C,x)y^{n_3}z^{n_4}&=0.
    \end{align*}
    For convenience, we define
    \begin{align*}
        G_n:=\sum_{(n_3,n_4)\in\Bar{\mathcal{I}}_2^{(n)}}p_{n_3,n_4}(C,x)y^{n_3}z^{n_4},\qquad\text{and}\qquad G_{n_3,n_4}:=p_{n_3,n_4}(C,x)y^{n_3}z^{n_4}.
    \end{align*}
    First, let us consider case (a). Then, by Lemma~\ref{lem:generalized:eigenvectors}, all non-zero $G_{n_3,n_4}$ are generalized eigenvectors of $\operatorname{ad}_x$ with rank $n_4+1$ to the eigenvalue $\lambda (n_3+n_4)$~\cite{Bronson:1991}. By linearity of $\operatorname{ad}_x$, it follows that all non-zero $G_n$ are generalized eigenvectors of rank at most $n+1$ to the eigenvalue $\lambda n$. Note that all $\lambda n$ are distinct because $\lambda\neq 0$. Recall that generalized eigenvectors to distinct eigenvalues are linearly independent~\cite{Axler:2023}. Thus, it follows that $G_n=0$ for all $n\leq N$. For each $n$, we can now compute
    \begin{align*}
        0&=(\operatorname{ad}_x-\lambda n\operatorname{id}_\mathcal{U})^nG_n=\sum_{(n_3,n_4)\in\Bar{\mathcal{I}}_2^{(n)}}p_{n_3,n_4}(C,x)(\operatorname{ad}_x-\lambda n\operatorname{id}_\mathcal{U})^ny^{n_3}z^{n_4}\\
        &=\sum_{(n_3,n_4)\in\Bar{\mathcal{I}}_2^{(n)}}p_{n_3,n_4}(C,x) n_4!\delta_{n_4,n}y^{n}=n!p_{0,n}(C,x)y^n.
    \end{align*}
    Since $\mathcal{U}$ is a domain by Lemma~\ref{lem:structural:lietarture:results}, it holds that $y^n\neq 0$, and consequently $p_{0,n}(C,x)=0$. This already settles the case $n=0$. For $n\geq1$, similarly,
    \begin{align*}
        0&=(\operatorname{ad}_x-\lambda n\operatorname{id}_\mathcal{U})^{n-1}G_n=n! p_{0,n}(C,x) y^{n-1}z+(n-1)!p_{1,n-1}(C,x)y^n.
    \end{align*}
    Thus, it must also hold that $p_{1,n-1}(C,x)=0$. Iterating this process yields consequently that $p_{n_3,n_4}(C,x)=0$ for all $n_3+n_4=n\leq N$.

    \vspace{0.2cm}

    Next, let us consider case (b). In this case, we have that
    \begin{align*}
        \operatorname{ad}_x\left(p_{n_3,n_4}(C,x)y^{n_3}z^{n_4}\right)=p_{n_3,n_4}(C,x)\left( \operatorname{ad}_x(y^{n_3})z^{n_4}+y^{n_3}\operatorname{ad}_x(z^{n_4})\right)=(\lambda n_3+\mu n_4)p_{n_3,n_4}(C,x)y^{n_3}z^{n_4},
    \end{align*}
    because $\operatorname{ad}_x:\mathcal{U}\to\mathcal{U}$ is a derivation. Thus all non-zero $G_{n_3,n_4}$ are eigenvectors of $\operatorname{ad}_x$ to the eigenvalues $(\lambda n_3+\mu n_4)$. However, since $\lambda/\mu\notin\Q$, these eigenvalues are distinct for all $(n_3,n_4)\in\mathcal{I}_2^{(N)}$. Thus, the nonzero $G_{n_3,n_4}$ are linearly independent, so Eq.~\ref{eqn:subcliam:cxyz:linear:independent:V2} implies $G_{n_3,n_4}=0$ for all $(n_3,n_4)\in\mathcal{I}_2^{(N)}$. Since $\mathcal{U}$ is a domain and $y\neq 0\neq z$, it follows that $p_{n_3,n_4}(C,x)=0$ in case (b).

    \vspace{0.2cm}
    
    However, in both cases (a) and (b), by Lemma~\ref{lem:polynom:C:x:zero:if:x:nonnilpotent}, $p_{n_3,n_4}(C,x)=0$ is possible if and only if $p_{n_3,n_4}=0$ for all such $n_3,n_4$, which implies that $\alpha_{\boldsymbol{n}}=0$ for all $\boldsymbol{n}\in\mathcal{I}_4^{(N)}$. This shows that $\mathcal{B}_N$ is linearly independent for all $N\in\Np{}$.
\end{proof}

\begin{lemma}\label{lem:adx:diagonalizable}
    Let $\g$ be a finite-dimensional non-nilpotent  solvable Lie subalgebra of $\mathcal{U}$. Moreover, let $x\in \g$ such that $\operatorname{ad}_x$ is not nilpotent. Then $\operatorname{ad}_x:\g\to\g$ is diagonalizable.
\end{lemma}

\begin{proof}
    We show this by contradiction and assume that $\operatorname{ad}_x$ is not diagonalizable. By the Jordan decomposition, the endomorphism $\operatorname{ad}_x$ is similar to an endomorphism on $\g$ that has Jordan block form. First, suppose $\operatorname{ad}_x$ has a generalized eigenspace to eigenvalue $\lambda\neq 0$ with non-trivial Jordan block structure~\cite{Axler:2023}. Then, there exist $y,z\in \g\setminus\{0\}$ such that
    \begin{align*}
        \operatorname{ad}_x(y)&=\lambda y,\;&\;\operatorname{ad}_x(z)&=\lambda z+y.
    \end{align*}
    Since $y=\lambda^{-1}[x,y]$ and $z=\lambda^{-2}[x,\lambda z-y]$, both $y$ and $z$ lie in $\mathcal{D}\g$. Corollary~\ref{cor:derived:solvable:algebra:abelian} now implies that $[y,z]=0$. It follows now from 
    Lemma~\ref{lem:Cxyz:linearly:indepndent} that, for all $N\in\Np{}$, the set 
    \begin{align*}
        \mathcal{B}_N:=\left\{C^{n_1}x^{n_2}y^{n_3}z^{n_4}\,\middle\mid\, (n_1,n_2,n_3,n_4)\in\mathcal{I}_4^{(N)}\right\}
    \end{align*}
    is linearly independent. We will now show that this is impossible. For this, set $d:=\max\{\deg(C),\deg(x),\deg(y),\deg(z)\}$ and consider the subspace
    \begin{align*}
        W_N:=\spn\{C^{n_1}x^{n_2}y^{n_3}z^{n_4}\,\mid\,(n_1,n_2,n_3,n_4)\in\mathcal{I}_4^{(N)}\}\subseteq\mathcal{U},
    \end{align*}
    for some $N\in\N_{\geq0}$. By Lemma~\ref{lem:Cxyz:linearly:indepndent}, the monomials $C^{n_1}x^{n_2}y^{n_3}z^{n_4}$ are all linearly independent. Thus,
    \begin{align*}
        \dim(W_N)&=\sum_{n=0}^N\dim\left(\spn\{C^{n_1}x^{n_2}y^{n_3}z^{n_4}\,\mid\,(n_1,n_2,n_3,n_4)\in\Bar{\mathcal{I}}_4^{(n)}\}\right)\\
        &=\sum_{n=0}^N\binom{n+3}{3}=\frac{(N+1)(N+2)(N+3)(N+4)}{24}\in\Theta(N^4),
    \end{align*}
    by the usual combinatorial argument that  distributing $n$ indistinguishable balls among $k$ distinguishable urns yields $(n+k-1)!/(n!(k-1)!)$ distinct possibilities~\cite{Stanley:2011,Wagner:2014}. However, by construction, it must also hold that
    \begin{align*}
        W_N\subseteq\spn\{X^{k_1}Y^{k_2}Z^{k_3}\,\mid\,(k_1,k_2,k_3)\in\mathcal{I}_3^{(dN)}\}=:V_N,
    \end{align*}
    where, by a similar argument as before, it must hold that
    \begin{align*}
        \dim(V_N)&=\sum_{n=0}^{dN}\dim\left(\spn\{X^{k_1}Y^{k_2}Z^{k_3}\,\mid\,(k_1,k_2,k_3)\in\Bar{\mathcal{I}}_3^{(n)}\}\right)\\
        &=\sum_{n=0}^{dN}\binom{n+2}{2}=\frac{(dN+1)(dN+2)(dN+3)}{6}\in\Theta(N^3).
    \end{align*}
    Since $W_N\subseteq V_N$, it must hold that $\dim(W_N)\leq \dim(V_N)$, which is incompatible with the obtained asymptotic bounds for sufficiently large $N$, since $d$ is fixed. Thus, the initial assumption must have been incorrect, and we can conclude that $\operatorname{ad}_x$ does not possess non-trivial Jordan blocks to a non-zero eigenvalue.

    \vspace{0.2cm}

    Next, let us assume $\operatorname{ad}_x$ has a generalized eigenspace with eigenvalue $\lambda_*=0$ and non-trivial Jordan block structure. This implies that there exist $y,z\in\g\setminus\{0\}$ such that $\operatorname{ad}_x(y)=0$ and $\operatorname{ad}_x(z)=y$. We assumed that $\operatorname{ad}_x$ is not nilpotent. Thus, there exists a non-zero eigenvalue $\lambda\in\C\setminus\{0\}$ and non-zero eigenvector $v\in\g\setminus\{0\}$ such that $\operatorname{ad}_x(v)=\lambda v\neq 0$. Thus, both $y=[x,z]$ and $v=\lambda^{-1}[x,v]$ belong to $\mathcal{D}\g$. Recall that we assume that $\g$ is solvable. Then, by Corollary~\ref{cor:derived:solvable:algebra:abelian} $\mathcal{D}\g$ is abelian. Hence $[v,y]=0$. It follows that $x,v\in C_\mathcal{U}(y)$. Supposing that $y\notin \mathcal{Z}(\mathcal{U})$ would enforce, by the acc of $\mathcal{U}$, that $[x,v]=0$, contradicting the earlier assumption that $v$ is a non-zero eigenvector of $\operatorname{ad}_x$ to a non-zero eigenvalue. Hence, $y\in \mathcal{Z}(\mathcal{U})$. However, $y=[x,z]\in[\mathcal{U},\mathcal{U}]$, and consequently, by Lemma~\ref{lem:structural:lietarture:results}, $0\neq y\in\mathcal{Z}(\mathcal{U})\cap[\mathcal{U},\mathcal{U}]=\{0\}$, which is impossible. Thus, the initial assumption is false. Together with the exclusion of non-trivial Jordan blocks at non-zero eigenvalues, it follows that $\operatorname{ad}_x:\g\to\g$ is diagonalizable.
\end{proof}

\begin{lemma}\label{lem:adx:positive:rational:ratios}
    Let $\g$ be a finite-dimensional non-nilpotent  solvable Lie subalgebra of $\mathcal{U}$, and let $x\in\g$ be such that $\operatorname{ad}_x:\g\to\g$ is not nilpotent. Then $\lambda/\mu\in\Q_{>0}$ for all $\lambda,\mu\in\operatorname{Spec}(\operatorname{ad}_x)\setminus\{0\}$.
\end{lemma}

\begin{proof}
    For ease of presentation, let us write $\boldsymbol{n}:=(n_1,n_2,n_3,n_4)\in\Np{4}$ and define the sets $\mathcal{S}:= \operatorname{Spec}(\operatorname{ad}_x)$ and $\mathcal{S}^*:=\mathcal{S}\setminus\{0\}$. Note that both sets are non-empty and $0\in\mathcal{S}$. Indeed, $\spn\{x\}\subseteq\ker(\operatorname{ad}_x)$ shows that $0\in\mathcal{S}$, and the fact that $\operatorname{ad}_x$ is not nilpotent implies that there exists a non-zero eigenvalue in $\mathcal{S}$~\cite{Axler:2023}.

    \vspace{0.2cm}

    We first want to show that the ratio of all non-zero eigenvalues of $\operatorname{ad}_x$ is a non-zero rational. We show this claim by contradiction, and assume that $y,z\in\g\setminus\{0\}$ are non-zero eigenvectors of $\operatorname{ad}_x$ to non-zero eigenvalues $\lambda$ and $\mu$, respectively, which satisfy $\lambda/\mu\notin \Q$. Here, it holds again that $[y,z]=0$ by Corollary~\ref{cor:derived:solvable:algebra:abelian} because $y=\lambda^{-1}[x,y]\in\mathcal{D}\g$ and $z=\mu^{-1}[x,z]\in\mathcal{D}\g$. Let $N\in\Np{}$, and consider a polynomial $q_N(C,x,y,z)$ of the form
    \begin{align*}
        q_N(C,x,y,z)=\sum_{\boldsymbol{n}\in\mathcal{I}_4^{(N)}}\alpha_{\boldsymbol{n}}C^{n_1}x^{n_2}y^{n_3}z^{n_4}=\sum_{(n_3,n_4)\in\mathcal{I}_2^{(N)}}p_{n_3,n_4}(C,x)y^{n_3}z^{n_4},
    \end{align*}
    where, as before in the proof of Lemma~\ref{lem:Cxyz:linearly:indepndent}, we defined
    \begin{align*}
        p_{n_3,n_4}(a,b):=\sum_{(n_1,n_2)\in\mathcal{I}_2^{(N-n_3-n_4)}}\alpha_{\boldsymbol{n}}a^{n_1}b^{n_2}\in\C[a,b].
    \end{align*}
    Lemma~\ref{lem:Cxyz:linearly:indepndent} implies here that the set 
    \begin{align*}
        \mathcal{B}_N:=\{C^{n_1}x^{n_2}y^{n_3}z^{n_4}\,\mid\,(n_1,n_2,n_3,n_4)\in\mathcal{I}_4^{(N)}\}
    \end{align*}
    is linearly independent for all $N\in\Np{}$. However, a similar argument as employed in Lemma~\ref{lem:adx:diagonalizable} shows that this is impossible. Indeed, set $d:=\max\{\deg(C),\deg(x),\deg(y),\deg(z)\}$. Every member of $\mathcal{B}_N$ belongs to the subspace
    \begin{align*}
        V_N:=\spn\left\{ X^{k_1}Y^{k_2}Z^{k_3}\,\mid\,(k_1,k_2,k_3)\in\mathcal{I}_3^{(dN)}\right\}.
    \end{align*}
    Linear independence would therefore imply
    \begin{align*}
        \binom{N+4}{4}\leq\dim(V_N)=\binom{dN+3}{3},
    \end{align*}
    which is impossible for all sufficiently large $N$. Thus, there must exist some $N\in \Np{}$ and coefficients $\alpha_{\boldsymbol{n}}$, not all zero, such that $q_N(C,x,y,z)=0$. We may now apply $\operatorname{ad}_x$ to this polynomial relation to find that
    \begin{align*}
        0&=\operatorname{ad}_x(q_N(C,x,y,z))=\sum_{(n_3,n_4)\in\mathcal{I}_2^{(N)}}p_{n_3,n_4}(C,x)(n_3\lambda+n_4\mu)y^{n_3}z^{n_4}.
    \end{align*}
    Each non-zero term $p_{n_3,n_4}(C,x)y^{n_3}z^{n_4}$ is an eigenvector of $\operatorname{ad}_x:\mathcal{U}\to\mathcal{U}$ to the eigenvalue $n_3\lambda+n_4\mu$. The assumption $\lambda/\mu \notin \Q$ now allows the equality $n_3\lambda+n_4\mu=n_3'\lambda+ n_4'\mu$ to hold if and only if $n_3=n_3'$ and $n_4=n_4'$. Thus the non-zero elements $p_{n_3,n_4}(C,x)y^{n_3}z^{n_4}$ are eigenvectors to distinct eigenvalues, and therefore linearly independent~\cite{Axler:2023}. Employing this linear independence along with the original relation $q_N(C,x,y,z)=0$, including the term with $n_3=n_4=0$, implies that $p_{n_3,n_4}(C,x)y^{n_3}z^{n_4}=0$ for all $n_3,n_4\in\Np{}$ with $n_3+n_4\leq N$. Since $\mathcal{U}$ is a domain and $y\neq 0\neq z$, it must follow that $p_{n_3,n_4}(C,x)=0$ for every such $(n_3,n_4)\in\mathcal{I}_2^{(N)}$. By Lemma~\ref{lem:polynom:C:x:zero:if:x:nonnilpotent}, this is only possible if $p_{n_3,n_4}=0$. Consequently, all coefficients $\alpha_{\boldsymbol{n}}$ vanish, a contradiction. Thus $\lambda/\mu\notin\Q$ is impossible.

    \vspace{0.2cm}

    In the second step, we want to show that the ratios of all non-zero eigenvalues $\lambda,\mu\in\mathcal{S}^*$ are furthermore positive rationals. To the contrary, suppose $\lambda/\mu<0$, and choose non-zero eigenvectors $y,z\in\g\setminus\{0\}$ with $[x,y]=\lambda y$ and $[x,z]=\mu z$. Then, there exist $r,s\in\N_{>0}$ such that $\lambda r+\mu s=0$, because $\lambda/\mu\in \Q$. Moreover, $y=\lambda^{-1}[x,y]\in\mathcal{D}\g$ and $z=\mu^{-1}[x,z]\in\mathcal{D}\g$. Hence, $[y,z]=0$ by Corollary~\ref{cor:derived:solvable:algebra:abelian}.  Set $u:= y^r z^s$. Then $\operatorname{ad}_x(u)=(\lambda r+\mu s)u=0$. Moreover, because $y,z\in \mathcal{D}\g$, $[y,z]=0$, and consequently $[y,u]=0=[z,u]$. It follows that $x,y,z\in C_\mathcal{U}(u)$. If $u$ were non-central with respect to $\mathcal{U}$, the acc of $\mathcal{U}$ from Lemma~\ref{lem:structural:lietarture:results} would imply that $[x,y]=0$, contradicting the assumption that $[x,y]=\lambda y\neq 0$. Hence $u\in \mathcal{Z}(\mathcal{U})$. However, it also holds that
    \begin{align*}
        \frac{1}{s\mu}[xy^r,z^s]=y^r z^s=\frac{1}{s\mu}y^r[x,z^s]=\frac{s\mu}{s\mu}y^rz^s=u\in[\mathcal{U},\mathcal{U}].
    \end{align*}
    Since $\mathcal{U}$ is a domain, $0\neq  u\in\mathcal{Z}(\mathcal{U})\cap[\mathcal{U},\mathcal{U}]=\{0\}$, a contradiction to Lemma~\ref{lem:structural:lietarture:results}. The initial assumption must therefore have been incorrect, and $\lambda/\mu \in\Q_{>0}$ for all $\lambda,\mu\in\mathcal{S}^*$.
\end{proof}

\begin{lemma}\label{lem:solvable:non:nilpotent}
    Let $\g$ be a finite-dimensional non-nilpotent  solvable Lie subalgebra of $\mathcal{U}$. Then $\g\cong\mathfrak{r}(\boldsymbol{d},\boldsymbol{m})\oplus\C^{d_0}$, for some $\boldsymbol{d},\boldsymbol{m}\in \N_{>0}^m$ and $d_0\in\Np{}$, where $m\in\N_{>0}$ and $d_j=d_k$ if and only if $j=k$.
\end{lemma}

\begin{proof}
    Let $\g\subseteq\mathcal{U}$ be a finite-dimensional non-nilpotent solvable Lie subalgebra of $\mathcal{U}$. Then, there exists an $x\in\g$ such that $\operatorname{ad}_x:\g\to\g$ is not nilpotent; otherwise $\g$ would be nilpotent by Engel's theorem~\cite{Humphreys:1972}. Lemmata~\ref{lem:adx:diagonalizable} and~\ref{lem:adx:positive:rational:ratios} now imply that $\operatorname{ad}_x$ is diagonalizable and the ratio of all non-zero eigenvalues is a positive rational number. For convenience, let us denote the set of eigenvalues by $\mathcal{S}_x:=\operatorname{Spec}(\operatorname{ad}_x)$ and the corresponding set of non-zero eigenvalues by $\mathcal{S}_x^*:=\mathcal{S}_x\setminus\{0\}$. Note that $0\in\mathcal{S}_x$ and $\mathcal{S}_x^*\neq\emptyset$ as explained in the proof of Lemma~\ref{lem:adx:positive:rational:ratios}. Moreover, let us write $\lambda_1,\ldots,\lambda_m$ for all distinct eigenvalues in $\mathcal{S}_x^*$.
    
    Since $\lambda_j/\lambda_k\in \Q_{>0}$ for all $j,k\in\{1,\ldots,m\}$, it holds that $\lambda_j/\lambda_1=q_j/p_j$ for $q_j,p_j\in\N_{>0}$. Thus,
    \begin{align*}
        \frac{\lambda_\ell}{\lambda_1}\prod_{j=1}^m p_j =q_\ell\prod_{\substack{j=1\\j\neq \ell}}^m p_j\in\N_{>0}\qquad\text{for all }\ell\in\{1,\ldots,m\}.
    \end{align*}
    Set $\kappa:=\lambda_1^{-1}\prod_{j=1}^m p_j\in\C\setminus\{0\}$. Then $\operatorname{ad}_{\kappa x}$ has the eigenvalues $0$ and $\kappa\lambda_j\in\N_{>0}$ for all $j\in\{1,\ldots,m\}$. Let us define $\Tilde{x}:=\kappa x/g$, where $g:=\operatorname{gcd}\{\mu\in \mathcal{S}^*_{\kappa x}\}$. Thus $\operatorname{ad}_{\Tilde{x}}$ has $m\geq 1$ distinct non-zero eigenvalues $\kappa\lambda_1/g,\ldots ,\kappa\lambda_m/g\in \N_{>0}$ such that
    \begin{align*}
        \operatorname{gcd}\{\mu\in \mathcal{S}^*_{\Tilde{x}}\}=\operatorname{gcd}\left\{\frac{\kappa\lambda_1}{g},\ldots,\frac{\kappa\lambda_m}{g}\right\}=1.
    \end{align*}
    In what follows, we denote the normalized non-zero eigenvalues of $\operatorname{ad}_{\Tilde{x}}$ by $d_1<d_2<\ldots<d_m$.
    
    \vspace{0.2cm}
    
    As a vector space, we may now decompose $\g$ as
    \begin{align*}
        \g=\bigoplus_{\lambda\in \mathcal{S}_{\Tilde{x}}}\g_\lambda,\qquad\text{where}\qquad\g_\lambda:=\{u\in\g:\operatorname{ad}_{\Tilde{x}}(u)=\lambda u\}.
    \end{align*}
    Since $m\geq 1$, it must hold that
    \begin{align*}
        \g_{>0}:=\bigoplus_{\lambda\in\mathcal{S}_{\Tilde{x}}^*}\g_\lambda\neq\{0\}.
    \end{align*}
    Moreover, because $\Tilde{x}\notin \mathcal{Z}(\g)$ and consequently $\Tilde{x}\notin\mathcal{Z}(\mathcal{U})$, the acc of $\mathcal{U}$, recalled in Lemma~\ref{lem:structural:lietarture:results}, implies that $C_\mathcal{U}(\Tilde{x})$ is abelian. Thus $\g_0=C_\g(\Tilde{x})\subseteq C_\mathcal{U}(\Tilde{x})$ is abelian. Let now $a\in \g_0$, and let $y\in \g_\lambda$ for some $\lambda\in \mathcal{S}_{\Tilde{x}}^*$. By the Jacobi identity, it must hold that
    \begin{align*}
        \operatorname{ad}_{\Tilde{x}}([a,y])=[\Tilde{x},[a,y]]=-[a,[y,\Tilde{x}]]-[y,[\Tilde{x},a]]=\lambda [a,y].
    \end{align*}
    Thus $[a,y]$ either vanishes or is an eigenvector of $\operatorname{ad}_{\Tilde{x}}$ to the eigenvalue $\lambda$, and in both cases $[a,y]\in \g_\lambda$.  Furthermore, $\operatorname{ad}_{\Tilde{x}}$ and $\operatorname{ad}_a$ commute, since, for every $z\in \g$, it holds that
    \begin{align*}
        \operatorname{ad}_{\Tilde{x}}\circ\operatorname{ad}_a(z)=[\Tilde{x},[a,z]]=-[a,[z,\Tilde{x}]]-[z,[\Tilde{x},a]]=[a,[\Tilde{x},z]]=\operatorname{ad}_a\circ\operatorname{ad}_{\Tilde{x}}(z).
    \end{align*}
    Fix one non-zero eigenvalue $\lambda\in\mathcal{S}_{\Tilde{x}}^*$. For every $a\in\g_0$, the endomorphism $\operatorname{ad}_a|_{\g_\lambda}$ has an eigenvector in $\g_\lambda$ because $\g_\lambda$ is a non-zero finite-dimensional complex vector space~\cite{Axler:2023}. Hence, there exists some $w\in\g_\lambda\setminus\{0\}$ and $\nu(a)\in\C$ such that $[a,w]=\nu(a) w$. This can be employed to define the element $b:=a-\nu(a)\Tilde{x}/\lambda\in\g_0$. Since $\g_0$ is abelian, $[\Tilde{x},b]=0$. Moreover, by construction, it holds that $[\Tilde{x},w]=\lambda w\neq 0$ and
    \begin{align*}
        [b,w]=[a,w]-\frac{\nu(a)}{\lambda}[\Tilde{x},w]=\nu(a)w-\nu(a) w=0.
    \end{align*}
    This shows that $\Tilde{x},w\in C_{\mathcal{U}}(b)$. If $b$ were non-central in $\mathcal{U}$, the acc of $\mathcal{U}$ would force $\Tilde{x}$ and $w$ to commute, which is false. Thus $b\in \mathcal{Z}(\mathcal{U})$, which implies furthermore that $b\in \mathcal{Z}(\g)$, since $b\in\mathcal{Z}(\mathcal{U})\cap \g$.  Every $a\in \g_0$ can therefore be written as $a=b+\nu(a)\Tilde{x}/\lambda$ with $b\in \mathcal{Z}(\g)$ and $\nu(a)/\lambda\in \C$. It follows that $\g_0=\spn\{\Tilde{x}\}\oplus\mathcal{Z}(\g)$. Indeed, let $z\in\mathcal{Z}(\g)$. Then we may write $z= z_0+\sum_{\lambda\in\mathcal{S}_{\Tilde{x}}^*} z_\lambda$, where $z_0\in \g_0$ and $z_\lambda\in \g_\lambda$ for each $\lambda\in \mathcal{S}_{\Tilde{x}}^*$. Since $z$ is assumed to be central and $\g_0$ is abelian, it follows that $0=[\Tilde{x},z]=\sum_{\lambda\in\mathcal{S}_{\Tilde{x}}^*} \lambda z_\lambda$. Since the non-zero $z_\lambda$, with $\lambda\in \mathcal{S}_{\Tilde{x}}^*$, are linearly independent as eigenvectors of $\operatorname{ad}_{\Tilde{x}}$ with distinct non-zero eigenvalues, it must hold that $z_\lambda=0$ for each $\lambda\in \mathcal{S}_{\Tilde{x}}^*$. Thus, $\mathcal{Z}(\g)\subseteq\g_0$. Moreover, also by construction, $\spn\{\Tilde{x}\}\cap\mathcal{Z}(\g)=\{0\}$, since $|\mathcal{S}_{\Tilde{x}}^*|\geq 1$.
    
    It holds that $\operatorname{ad}_{\Tilde{x}}(\g_\lambda)=\g_\lambda$ for all $\lambda\in \mathcal{S}_{\Tilde{x}}^*$ and $\operatorname{ad}_{\Tilde{x}}(\g_0)=\{0\}$. Thus, $\g_{>0}\subseteq\mathcal{D}\g$. Since $\mathcal{D}\g$ is abelian by Corollary~\ref{cor:derived:solvable:algebra:abelian}, it follows that $[\g_{>0},\g_{>0}]=\{0\}$. Consequently, one has
    \begin{align*}
        \mathcal{D}\g=[\spn\{\Tilde{x}\}\oplus\mathcal{Z}(\g)\oplus\g_{>0},\spn\{\Tilde{x}\}\oplus\mathcal{Z}(\g)\oplus\g_{>0}]=[\spn\{\Tilde{x}\},\g_{>0}]=\g_{>0}.
    \end{align*}
    For $k\in\{1,\ldots,m\}$, set $m_k:=\dim(\g_{d_k})\in\N_{>0}$ and choose a basis $\{b_{k1},\ldots,b_{km_k}\}$ of each $\g_{d_k}$. Furthermore, set $d_0:=\dim(\mathcal{Z}(\g))\in\Np{}$. This shows that
    \begin{align*}
        \g=\left(\spn\{\Tilde{x}\}\ltimes\left(\bigoplus_{k=1}^m\spn\{b_{kj}\,\mid\, j\in\{1,\ldots,m_k\}\}\right)\right) \oplus\mathcal{Z}(\g),
    \end{align*}
    where $[\Tilde{x},b_{kj}]=d_{k} b_{kj}$, and $[b_{kj},b_{k'j'}]=0$. That is,
    \begin{align*}
        \g\cong\mathfrak{r}(\boldsymbol{d},\boldsymbol{m})\oplus\C^{d_0}\qquad\text{with}\qquad\boldsymbol{d}=(d_1,\ldots,d_m)^\mathrm{Tp}\in\N_{>0}^m,\;\boldsymbol{m}=(m_1,\ldots,m_m)\in\N_{>0}^m,
    \end{align*}
    and $d_j=d_k$ if and only if $j=k$, as well as $\operatorname{gcd}(d_1,\ldots,d_m)=1$.
\end{proof}

\subsection{Realizations}

In the final part of the proof of Theorem~\ref{thm:classification:Usl2C}, we want to show that every Lie algebra listed there can be faithfully realized within $\mathcal{U}$ as a Lie subalgebra.

\begin{proposition}
    Any finite-dimensional abelian Lie algebra can be faithfully realized as a Lie subalgebra of $\mathcal{U}$.
\end{proposition}

\begin{proof}
    Let $m\in \N_{\geq0}$. The powers $X^1,\ldots,X^m$ commute pairwise and are linearly independent by the PBW theorem. Thus $\C^m\cong \spn\{X^j\,\mid\, j\in\{1,\ldots,m\}\}\subseteq\mathcal{U}$ as a Lie algebra with the trivial Lie bracket. Every finite-dimensional abelian Lie algebra is isomorphic to $\C^m$ for its dimension $m$, where we use the convention that $\C^0=\{0\}=\spn\{\emptyset\}$.
\end{proof}

\begin{proposition}
    The Lie algebra $\sl{2}{\C}$ can be faithfully realized within $\mathcal{U}$.
\end{proposition}

\begin{proof}
    By construction, it holds that $\sl{2}{\C}\cong\spn\{X,Y,Z\}\subseteq\mathcal{U}$.
\end{proof}

However, it is important to observe that the natural embedding $\iota:\sl{2}{\C}\hookrightarrow\mathcal{U}$ (cf. Ref.~\cite{Onishchik:1988}) is not the only embedding up to precomposition by automorphisms of $\sl{2}{\C}$, as demonstrated in the following example.

\begin{tcolorbox}
    \begin{example}\label{exa:strange:sl2:embedding}
        Define the linear map $\varphi:\sl{2}{\C}\to\mathcal{U}$ by
    \begin{align*}
        \varphi(E)&:=E+\frac{1}{4}(FH+HF)-\frac{1}{4}F^3,\;&\;\varphi(F)&:=F,\;&\;\varphi(H):=H-F^2.
    \end{align*}
    It is straightforward to verify that $[\varphi(H),\varphi(E)]=2\varphi(E)$, $[\varphi(H),\varphi(F)]=-2\varphi(F)$, and $[\varphi(E),\varphi(F)]=\varphi(H)$, and hence $\varphi:\sl{2}{\C}\hookrightarrow\mathcal{U}$, while clearly $\varphi(\sl{2}{\C})\neq\spn\{X,Y,Z\}=\iota(\sl{2}{\C})$.
    \end{example}
\end{tcolorbox}

It is straightforward to extend the construction employed to obtain the embedding presented in Example~\ref{exa:strange:sl2:embedding} to create embeddings that use terms of arbitrarily high degree. 

\begin{proposition}\label{eqn:non:linear:realizations:sl2C}
    Let $m\in\N_{>1}$. Then the linear map $\varphi_m:\sl{2}{\C}\to \mathcal{U}$, defined via
    \begin{align*}
        \varphi_m(E)&:=E+\frac{1}{2m}\sum_{\ell=0}^{m-1}F^\ell H F^{m-1-\ell}-\frac{1}{4}F^{2m-1},\;&\; \varphi_m(F)&:=F,\;&\;\varphi_m(H)&:=H-F^m,
    \end{align*}
    constitutes an embedding of $\sl{2}{\C}$ in $\mathcal{U}$.
\end{proposition}

\begin{proof}
    One computes
    \begin{align*}
        [\varphi_m(H),\varphi_m(E)]&=[H,E]+\frac{1}{2m}\sum_{\ell=0}^{m-1}\left([H,F^\ell]H F^{m-1-\ell}+F^\ell H [H,F^{m-1-\ell}]\right)-\frac{1}{4}[H,F^{2m-1}]\\
        &\quad-[F^m,E]-\frac{1}{2m}\sum_{\ell=0}^{m-1}F^\ell [F^m,H]F^{m-1-\ell}+\frac{1}{4}[F^m,F^{2m-1}]\\
        &=2E+\frac{1}{2m}\sum_{\ell=0}^{m-1}(-2\ell-2(m-1-\ell))F^\ell H F^{m-1-\ell}+\frac{2m-1}{2}F^{2m-1}+\sum_{\ell=0}^{m-1}F^\ell H F^{m-1-\ell}-m F^{2m-1}\\
        &=2\left(E+\frac{1}{2m}\sum_{\ell=0}^{m-1}F^\ell H F^{m-1-\ell}-\frac{1}{4}F^{2m-1}\right)=2\varphi_m(E),\\
        [\varphi_m(H),\varphi_m(F)]&=[H,F]-[F^m,F]=-2F=-2\varphi_m(F),\\
        [\varphi_m(E),\varphi_m(F)]&=[E,F]+\frac{1}{2m}\sum_{\ell=0}^{m-1}F^\ell[H,F]F^{m-1-\ell}=H-F^m=\varphi_m(H),
    \end{align*}
    which shows that $\varphi_m$ is a Lie-algebra homomorphism. Its injectivity follows directly from a close inspection of the degrees of $\varphi_m(H)$, $\varphi_m(E)$, and $\varphi_m(F)$.
\end{proof}

\begin{corollary}
    The Lie algebra $\sl{2}{\C}\oplus\C^{d_0}$ can be faithfully realized within $\mathcal{U}$ for all $d_0\in\N_{>0}$.
\end{corollary}

\begin{proof}
    Consider the space $\spn\{X,Y,Z\}\oplus\spn\{C^k\,\,\mid\, k\in\{1,\ldots,d_0\}\}\subseteq\mathcal{U}$. The second summand is central as it lies within $\mathcal{Z}(\mathcal{U})=\C[C]$, while the first summand realizes the special linear Lie algebra $\sl{2}{\C}$. The set $\{C^k\,\,\mid\, k\in\{1,\ldots,d_0\}\}$ is linearly independent, as the elements have strictly increasing degrees $2,4,\ldots,2d_0$. Furthermore the terms in the first and in the second summand are linearly independent, as the former contains the three degree-one-basis elements generating $\mathcal{U}$, while the latter terms contain only elements of even degree. The sum is therefore direct and satisfies the required bracket structure.
\end{proof}

\begin{proposition}
    Let $m\in\N_{>0}$, $\boldsymbol{d},\boldsymbol{m}\in \N_{>0}^m$ and $d_0\in\Np{}$ such that the components $d_j$ of $\boldsymbol{d}$ are distinct, correctly ordered and have greatest common divisor one. Then, the Lie algebra $\mathfrak{r}(\boldsymbol{d},\boldsymbol{m})\oplus\C^{d_0}$ can be faithfully realized within $\mathcal{U}$.
\end{proposition}

\begin{proof}
    It holds that $[-\mathrm{i}X,Z+\mathrm{i}Y]=Z+\mathrm{i}Y$, and consequently $[-\mathrm{i}X,(Z+\mathrm{i}Y)^k]=k(Z+\mathrm{i}Y)^k$. Thus, since $C$ is central, it holds that
    \begin{align*}
        [-\mathrm{i}X,(Z+\mathrm{i}Y)^kC^\ell]=k(Z+\mathrm{i}Y)^k C^\ell,\qquad\text{while}\qquad[(Z+\mathrm{i}Y)^kC^\ell,(Z+\mathrm{i}Y)^{k'}C^{\ell'}]=0\text{ for all }k,\ell,k',\ell'\in\Np{}.
    \end{align*}
    This shows that the required bracket structure holds on
    \begin{align*}
        \mathcal{B}:=\left\{-\mathrm{i}X\right\}\cup\left\{(Z+\mathrm{i}Y)^{d_j}C^{\ell_j}\,\mid\, j\in\{1,\ldots,m\},\,\ell_j\in\{1,\ldots,m_j\}\right\}\cup\left\{C^{k}\,\mid\,  k\in\{1,\ldots,d_0\}\right\}.
    \end{align*}
    It remains to show that $\mathcal{B}$ is linearly independent. The elements in $\mathcal{B}\setminus\{-\mathrm{i}X\}$ are all eigenvectors of $\operatorname{ad}_{-\mathrm{i}X}$ with eigenvalues $0,d_1,\ldots,d_m$. Eigenvectors belonging to distinct eigenvalues are linearly independent. Thus, linear dependence can only occur for the sets $\{-\mathrm{i}X\}\cup\{C^k\}_{k=1}^{d_0}$ and $\{ (Z+\mathrm{i}Y)^{d_j} C^{\ell_j}\}_{\ell_j=1}^{m_j}$, with $j\in\{1,\ldots,m\}$. The zero-eigenvalue group consists of $-\mathrm{i}X$ and the powers $C^k$, which are linearly independent because their degrees 1 and $2k$, with $k\in\{1,\ldots,d_0\}$, are distinct. Finally, suppose
    \begin{align*}
        \sum_{\ell_j=1}^{m_j} \alpha_{\ell_j} (Z+\mathrm{i}Y)^{d_j} C^{\ell_j}=0.
    \end{align*}
    Since, for all $j\in\{1,\ldots,m\}$,
    \begin{align*}
        \deg\left(\sum_{\ell_j=1}^{m_j-1} \alpha_{\ell_j} (Z+\mathrm{i}Y)^{d_j} C^{\ell_j}\right)\leq d_j+2(m_j-1)<d_j+2m_j=\deg\left((Z+\mathrm{i}Y)^{d_j} C^{m_j}\right),
    \end{align*}
    which can be seen by expanding the terms in PBW normal order, it must hold that $\alpha_{m_j}=0$. We can now iteratively repeat this process for the next highest degree term and can therefore conclude that $\alpha_{\ell_j}=0$ for all $\ell_j$. Thus, each set $\{ (Z+\mathrm{i}Y)^{d_j} C^{\ell_j}\}_{\ell_j=1}^{m_j}$ is linearly independent. This completes the proof, as it shows that $\spn\{\mathcal{B}\}\cong \mathfrak{r}(\boldsymbol{d},\boldsymbol{m})\oplus\C^{d_0}$.
\end{proof}

\subsection{Classification for the skew-hermitian case}\label{sec:classification:of:skew:hermitian:case}

Let us equip $\mathcal{U}$ with the conjugate-linear anti-involution $(\cdot)^\dagger:\mathcal{U}\to\mathcal{U}$ defined by $X^\dagger=-X$, $Y^\dagger=-Y$, $Z^\dagger=-Z$, $\unity^\dagger=\unity$, $(\mathrm{i}\unity)^\dagger=-\mathrm{i}\unity$, extended conjugate-linearly and by the rule $(AB)^\dagger=B^\dagger A^\dagger$. We want to denote this operation as a \enquote{hermitian conjugation}. This allows us to define a real \enquote{skew-hermitian} Lie subalgebra of $\mathcal{U}$:
\begin{align*}
    \hat{\mathcal{U}}:=\left\{A\in \mathcal{U}\,\mid\, A^\dagger =-A\right\}.
\end{align*}

Here, we find the following classification result.

\begin{thmbox}[boxsep=0pt,left=1em,right=1em]{}
    \begin{corollary}[Classification result]\label{thm:classification:hat:Usl2C}
    Let $\g$ be a real non-zero finite-dimensional Lie subalgebra of $\hat{\mathcal{U}}$. Then, $\g$ is isomorphic to exactly one of the following Lie algebras:
    \begin{itemize}
        \item If solvable: $\g\cong \R^d$ for some $d\in\N_{>0}$.
        \item If semisimple: $\g\cong\su_\R(2)$.
        \item If non-semisimple and non-solvable: $\g\cong\su_\R(2)\oplus\R^d$ for some $d\in\N_{>0}$.
    \end{itemize}
    \end{corollary}
\end{thmbox}

\begin{lemma}\label{lem:faithful:skew:hermitian:representation}
    Let $\g$ be a real non-zero finite-dimensional Lie subalgebra of $\hat{\mathcal{U}}$. Then, $\g$ admits a faithful finite-dimensional representation in terms of skew-hermitian matrices.
\end{lemma}

\begin{proof}
    Let us consider the representation $\rho:\sl{2}{\C}\to \operatorname{End}(\C[x_1,x_2])$ of $\sl{2}{\C}$ given by
    \begin{align*}
        \rho(X)&:=-\frac{\mathrm{i}}{2}\left(x_1\partial_2+x_2\partial_1\right),\;&\;\rho(Y)&:=\frac{1}{2}\left(x_1\partial_2-x_2\partial_1\right),\;&\;\rho(Z)&:=\frac{\mathrm{i}}{2}\left(x_1\partial_1-x_2\partial_2\right),
    \end{align*}
    where $\partial_j:=\partial/\partial x_j$. It is straightforward to confirm that these operators satisfy the required Lie-algebra homomorphism requirements
    \begin{align*}
        [\rho(X),\rho(Y)]&=\rho(Z),\;&\;[\rho(Y),\rho(Z)]&=\rho(X),\;&\;[\rho(Z),\rho(X)]&=\rho(Y).
    \end{align*}
    We can extend $\rho$ uniquely to a representation $P:\mathcal{U}\to \operatorname{End}(\C[x_1,x_2])$ using the universal property of the universal enveloping algebra. We may decompose $\C[x_1,x_2]$ into the spaces of homogeneous polynomials
    \begin{align*}
        \mathcal{P}_n:=\spn\left\{x_1^\ell x_2^{n-\ell}\,\middle\mid\,\ell\in\{0,1,\ldots,n\}\right\}\qquad\text{such that}\qquad\C[x_1,x_2]=\bigoplus_{n=0}^\infty\mathcal{P}_n.
    \end{align*}
    Each of the three operators $\rho(X)$, $\rho(Y)$, and $\rho(Z)$ preserves the degree of a polynomial. Hence, $P(A)\mathcal{P}_n\subseteq\mathcal{P}_n$, for all $A\in \mathcal{U}$. Thus, we may consider the restricted representations $P_n:\mathcal{U}\to \operatorname{End}(\mathcal{P}_n)$, where $U\mapsto P_n(U):=P(U)|_{\mathcal{P}_n}$. These are known as the canonical spin-$n/2$ representations~\cite{Heib:2026:SSRC:Universality,Georgi:2000}.
    
    Let now $U\in \g\setminus\{0\}\subseteq\mathcal{U}$, set $d:=\deg(U)$ and write $\boldsymbol{n}:=(n_1,n_2,n_3)\in\Np{3}$. By the PBW theorem, we can expand $U$ in terms of the ordered monomials in $X,Y,Z$ such that
    \begin{align*}
        U=\sum_{\boldsymbol{n}\in\mathcal{I}_3^{(d)}}\alpha_{\boldsymbol{n}}X^{n_1}Y^{n_2}Z^{n_3}=\sum_{\ell=0}^d\sum_{\boldsymbol{n}\in\Bar{\mathcal{I}}_3^{(\ell)}}\alpha_{\boldsymbol{n}}X^{n_1}Y^{n_2}Z^{n_3}.
    \end{align*}
    Since $\deg(U)=d$, it holds that
    \begin{align*}
        q_d(a,b,c):=\sum_{\boldsymbol{n}\in\Bar{\mathcal{I}}_3^{(d)}}\alpha_{\boldsymbol{n}}a^{n_1}b^{n_2}c^{n_3}\neq 0,
    \end{align*}
    as a polynomial. We now want to show that for some sufficiently large $n\in\Np{}$, it holds that $P_n(U)$ is not the zero operator, i.e., $\ker(P_n(U))\neq\mathcal{P}_n$. First, we write out $P(U)$ in terms of the variables $x_1,x_2$ and the partial derivatives $\partial_1,\partial_2$ using the commutation relations $[x_j,\partial_k]=-\delta_{jk}\operatorname{id}$:
    \begin{align*}
        P(U)&=\sum_{\boldsymbol{n}\in\mathcal{I}_3^{(d)}}\alpha_{\boldsymbol{n}}\left(-\frac{\mathrm{i}}{2}\left(x_1\partial_2+x_2\partial_1\right)\right)^{n_1}\left(\frac{1}{2}\left(x_1\partial_2-x_2\partial_1\right)\right)^{n_2}\left(\frac{\mathrm{i}}{2}\left(x_1\partial_1-x_2\partial_2\right)\right)^{n_3}=\sum_{(k_1,k_2)\in\mathcal{I}_2^{(d)}}p_{k_1k_2}(x_1,x_2)\partial_1^{k_1}\partial_2^{k_2}.
    \end{align*}
    Next, we introduce the commuting variables $\xi_1$ and $\xi_2$, which we substitute for the partial derivatives present in $P(U)$. 
    When commuting a derivative $\partial_j$ through a variable $x_k$, every non-zero correction term has fewer derivative factors, and hence lower derivative order, where the derivative order is the total number of derivative operators in a normal ordered monomial. Thus, the terms of derivative order $d$ satisfy
    \begin{align*}
        \sum_{(k_1,k_2)\in\Bar{\mathcal{I}_2^{(d)}}}p_{k_1k_2}(x_1,x_2)\xi_1^{k_1}\xi_2^{k_2}=q_d\left(-\frac{\mathrm{i}}{2}(x_1\xi_2+x_2\xi_1),\frac{1}{2}(x_1\xi_2-x_2\xi_1),\frac{\mathrm{i}}{2}(x_1\xi_1-x_2\xi_2)\right).
    \end{align*}
    Suppose now that the polynomial on the right-hand-side above vanishes identically and let $\lambda,\mu,\nu\in\C$ be arbitrary with $\mu-\mathrm{i}\lambda\neq 0$. We now seek to find values for the variables $x_1,x_2,\xi_1,\xi_2$ such that the three inputs of the polynomial on the right-hand-side above form the triple $(\lambda,\mu,\nu)$. That is, we want to solve the following set of equations:
    \begin{align*}
        \lambda&:= -\frac{\mathrm{i}}{2}(x_1\xi_2+x_2\xi_1),\;&\;\mu&:=\frac{1}{2}(x_1\xi_2-x_2\xi_1),\;&\;\nu&:=\frac{\mathrm{i}}{2}(x_1\xi_1-x_2\xi_2).
    \end{align*}
    We can now manipulate these expressions to obtain the equivalent relations
    \begin{align*}
        \mu+\mathrm{i}\lambda&=x_1\xi_2,\;&\;-\mu+\mathrm{i}\lambda&=x_2\xi_1,\;&\;-2\mathrm{i}\nu&=x_1\xi_1-x_2\xi_2.
    \end{align*}
    To construct a solution, we may start by choosing $x_2=1$. This allows us to conclude that $\xi_1=-\mu+\mathrm{i}\lambda$ and $\xi_2=2\mathrm{i}\nu-(\mu-\mathrm{i}\lambda) x_1$. Substituting this in the relation for $x_1\xi_2$ yields
    \begin{align*}
        0&=-(\mu-\mathrm{i}\lambda)x_1^2+2\mathrm{i}\nu x_1-(\mu+\mathrm{i}\lambda).
    \end{align*}
    The right-hand-side of the equation above is a polynomial of degree $2$ in $x_1$, since $\mu-\mathrm{i}\lambda\neq 0$. It has therefore a root in $\C$. Choosing $x_1$ accordingly, and the values for $\xi_1$ and $\xi_2$ as demanded by the equations above solves all three equations. Thus, every triple $(\lambda,\mu,\nu)\in\C^{3}$ with $\lambda\neq-\mathrm{i}\mu$ appears in the argument of $q_d$. Thus, one must have, for all such $\lambda,\mu,\nu\in \C$,
    \begin{align*}
        q_d(\lambda,\mu,\nu)=\sum_{\boldsymbol{n}\in\Bar{\mathcal{I}}_3^{(d)}} \alpha_{\boldsymbol{n}}\lambda^{n_1}\mu^{n_2}\nu^{n_3}=\sum_{k=0}^d f_k(\mu,\nu)\lambda^k=0.
    \end{align*}
    For each fixed $\mu,\nu\in\C$, this holds for all $\lambda\neq-\mathrm{i}\mu$. The polynomial above, seen only as a finite polynomial in $\lambda$, has consequently infinitely many zeros, which implies that $f_k(\mu,\nu)=0$, for every $k\in\{0,1,\ldots,d\}$. Since $\mu,\nu$ were chosen arbitrarily, we may repeat this argument in the remaining variables and find that $f_k=0$ as a polynomial. This would imply that $q_d=0$ as a polynomial, contradicting its construction.

    By the argument above $p_{k_1k_2}(x_1,x_2)\neq 0$ as a polynomial for at least one pair $(k_1,k_2)\in{\mathcal{I}}_2^{(d)}$. Choose now a pair $(k_1',k_2')\in{\mathcal{I}}_2^{(d)}$ with $p_{k_1'k_2'}\neq 0$ such that $k_1'+k_2'$ is minimal among all such pairs. It then holds that $\partial_1^{\ell_1}\partial_2^{\ell_2}x_1^{j_1}x_2^{j_2}=0$ whenever $\ell_1> j_1$ or $\ell_2>j_2$. Thus, by minimality, one has
    \begin{align}
        P(U)x_1^{k_1'}x_2^{k_2'}=k_1'!k_2'! p_{k_1'k_2'}(x_1,x_2)\neq0.
    \end{align}
    It holds that $x_1^{k_1'}x_2^{k_2'}\in \mathcal{P}_{k_1'+k_2'}$, where $k_1'+k_2'\leq d$. This shows that ${P}_n(U)$ is not the zero operator for at least one $n\in\{0,1,\ldots,d\}$.

    We may now choose a basis $\mathcal{B}=\{U_1,\ldots,U_m\}$ of $\g$ and set $D:=\max\{\deg(U_j)\,\mid\,j\in\{1,\ldots,m\}\}$. By the observations above, the direct sum of representations
    \begin{align*}
        \Phi:\mathcal{U}\longrightarrow \operatorname{End}\left(\bigoplus_{n=0}^D\mathcal{P}_n\right),\quad U\longmapsto \Phi(U):=\bigoplus_{n=0}^DP_n(U)
    \end{align*}
    satisfies $\ker(\Phi)\cap \g=\{0\}$, since each $U\in \g\setminus\{0\}$ has degree $\deg(U)\leq D$ and there exists some $n\in\{0,1,\ldots,D\}$ such that $P_n(U)$ is not the zero operator. Thus, the restriction $\Phi|_\g$ is a faithful finite-dimensional representation.

    To show that $\g$ admits such a faithful finite-dimensional representation in terms of skew-hermitian matrices, we equip $\C[x_1,x_2]$ with the scalar product defined on the basis elements via $\langle x_1^{\ell_1} x_2^{\ell_2},x_1^{k_1}x_2^{k_2}\rangle=\ell_1!\ell_2!\delta_{\ell_1k_1}\delta_{\ell_2 k_2}$, extended sesquilinearly. This allows us to define, for each $n\in\Np{}$, the orthonormal basis of $\mathcal{P}_n$ via
    \begin{align*}
        \mathcal{B}_n:=\left\{\ket{\ell}\equiv\frac{x_1^\ell x_2^{n-\ell}}{\sqrt{\ell!(n-\ell)!}}\,\middle\mid\, \ell\in\{0,1,\ldots,n\}\right\},
    \end{align*}
    with the convention $\ket{k}=0$ for all $k\in\Z\setminus\{0,1,\ldots,n\}$. This also allows us to identify $\mathcal{P}_n\cong \C^{n+1}$, as vector spaces. Writing $|j\rangle \langle k|$ as the elementary matrix with a one at position $(j,k)$ while all other entries are identically zero, we can express the representations of the basis elements $X,Y,Z$ as
    \begin{align*}
        P_n(X)&=-\frac{\mathrm{i}}{2}\sum_{k=0}^n\left(\sqrt{(k+1)(n-k)}\ket{k+1}\bra{k}+\sqrt{k(n-k+1)}\ket{k-1}\bra{k}\right),\\
        P_n(Y)&=\frac{1}{2}\sum_{k=0}^n\left(\sqrt{(k+1)(n-k)}\ket{k+1}\bra{k}-\sqrt{k(n-k+1)}\ket{k-1}\bra{k}\right),\\
        P_n(Z)&=\mathrm{i}\sum_{k=0}^n\left(k-\frac{n}{2}\right)\ket{k}\bra{k}.
    \end{align*}
    These are clearly skew-hermitian matrices. That is, $(P_n(A))^\dagger=-P_n(A)=P_n(-A)=P_n(A^\dagger)$ for each $A\in\{X,Y,Z\}$. Let now $U\in \hat{\mathcal{U}}$, with the usual decomposition. Then, one has
    \begin{align*}
        U^\dagger&=\left(\sum_{\boldsymbol{n}\in \mathcal{I}_3^{(d)}}\alpha_{\boldsymbol{n}} X^{n_1}Y^{n_2}Z^{n_3}\right)^\dagger=\sum_{\boldsymbol{n}\in \mathcal{I}_3^{(d)}}\alpha_{\boldsymbol{n}}^*(Z^\dagger)^{n_3} (Y^\dagger)^{n_2}(X^\dagger)^{n_1}=\sum_{\boldsymbol{n}\in \mathcal{I}_3^{(d)}}(-1)^{n_1+n_2+n_3}\alpha_{\boldsymbol{n}}^* Z^{n_3}Y^{n_2}X^{n_1}=-U,
    \end{align*}
    by the defining properties for the operation $(\cdot)^\dagger$ on $\mathcal{U}$. Thus, by the definition of $P_n$, one has furthermore
    \begin{align*}
        (P_n(U))^\dagger&= \sum_{\boldsymbol{n}\in\mathcal{I}_3^{(d)}}\alpha_{\boldsymbol{n}}^*\left((P_n(Z))^\dagger\right)^{n_3}\left((P_n(Y))^\dagger\right)^{n_2}\left((P_n(X))^\dagger\right)^{n_1}\\
        &=\sum_{\boldsymbol{n}\in\mathcal{I}_3^{(d)}}(-1)^{n_1+n_2+n_3}\alpha_{\boldsymbol{n}}^*\left(P_n(Z)\right)^{n_3}\left(P_n(Y)\right)^{n_2}\left(P_n(X)\right)^{n_1}\\
        &=P_n\left(\sum_{\boldsymbol{n}\in \mathcal{I}_3^{(d)}}(-1)^{n_1+n_2+n_3}\alpha_{\boldsymbol{n}}^* Z^{n_3}Y^{n_2}X^{n_1}\right)=P_n(-U)=-P_n(U).
    \end{align*}
    Thus, every $\Phi(U)$ is a finite-dimensional skew-hermitian block-diagonal matrix for every $U\in\g$, which completes the proof.
\end{proof}

This allows us now to prove Corollary~\ref{thm:classification:hat:Usl2C}.

\begin{proof}
    First, we need to observe that the real Lie algebra $\hat{\mathcal{U}}$ is constructed such that $\hat{\mathcal{U}}\oplus \mathrm{i}\hat{\mathcal{U}}=\mathcal{U}$ as a direct sum of real vector spaces, because every $A\in \mathcal{U}$ decomposes as the sum of a real skew-hermitian and real hermitian component: $A=(A-A^\dagger)/2+\mathrm{i}(A+A^\dagger)/(2\mathrm{i})$, while the intersection of $\hat{\mathcal{U}}$ and $\mathrm{i}\hat{\mathcal{U}}$ is trivial. We may consequently define the complexification of $\g\subseteq\hat{\mathcal{U}}$ via $\g_\C:=\g\oplus\mathrm{i}\g$, which is a complex Lie subalgebra of $\mathcal{U}$.

    Let us suppose $\g\subseteq\hat{\mathcal{U}}$ is a non-zero finite-dimensional real Lie subalgebra. According to Lemma~\ref{lem:faithful:skew:hermitian:representation}, $\g$ admits a faithful finite-dimensional representation $\Phi$. This allows us to equip $\g$ with an inner product, namely the induced Frobenius inner product, defined by
    \begin{align*}
        \langle\cdot,\cdot\rangle_{\mathrm{IF}}:\g\times\g\longrightarrow \R,\qquad\langle a,b\rangle_\mathrm{IF}:=\Tr{\Phi(a)(\Phi(b))^\dagger}.
    \end{align*}
    Since $\Phi$ is compatible with the skew hermitian conjugation on $\mathcal{U}$, it holds for all $a,b\in\g\subseteq\hat{\mathcal{U}}$ that
    \begin{align*}
        \langle a,b\rangle_{\mathrm{IF}}=\Tr{\Phi(a)(\Phi(b))^\dagger}=\Tr{\Phi(a)\Phi(b^\dagger)}=\Tr{\Phi(a)\Phi(-b)}=-\Tr{\Phi(a)\Phi(b)}.
    \end{align*}
    This binary operation constitutes indeed an inner product, as it inherits its conjugate symmetry and real bilinearity from the Frobenius inner product for complex matrices~\cite{AmirMoez:1960}. Moreover, it is real-valued and  hence symmetric, because
    \begin{align*}
        \langle a,b\rangle_\mathrm{IF}^*&=-\Tr{\Phi(a)\Phi(b)}^*=-\Tr{(\Phi(a)\Phi(b))^\dagger}=-\Tr{(\Phi(b))^\dagger(\Phi(a))^\dagger}=-\Tr{\Phi(b^\dagger)\Phi(a^\dagger)}=-\Tr{\Phi(b)\Phi(a)}\\
        &=-\Tr{\Phi(a)\Phi(b)}=\langle a,b\rangle_\mathrm{IF}.
    \end{align*}
    Moreover, since $\Phi$ is faithful, it also inherits the positive definiteness from the Frobenius inner product, making $\langle\cdot,\cdot\rangle_\mathrm{IF}$ an inner product on $\g$~\cite{Horn:2012,AmirMoez:1960,Axler:2023}.

    We may now observe that, using linearity and cyclicity of the trace and the fact that $\Phi$ is a Lie-algebra homomorphism, one has, for all $x,a,b\in\g$,
    \begin{align*}
        \langle\operatorname{ad}_x(a),b\rangle_\mathrm{IF}&=-\Tr{\Phi([x,a])\Phi(b)}=-\Tr{[\Phi(x),\Phi(a)]\Phi(b)}=-\Tr{\Phi(x)\Phi(a)\Phi(b)}+\Tr{\Phi(a)\Phi(x)\Phi(b)}\\
        &=-\Tr{\Phi(a)\Phi(b)\Phi(x)}+\Tr{\Phi(a)\Phi(x)\Phi(b)}=-\Tr{\Phi(a)[\Phi(b),\Phi(x)]}=-\Tr{\Phi(a)\Phi([b,x])}\\
        &=\Tr{\Phi(a)\Phi([x,b])}=-\langle a,\operatorname{ad}_x(b)\rangle_\mathrm{IF}.
    \end{align*}
    This shows that $\operatorname{ad}_x$ is skew-adjoint with respect to $\langle\cdot,\cdot\rangle_\mathrm{IF}$ for all $x\in\g\subseteq\hat{\mathcal{U}}$.

    Let us now assume that $\g$ is solvable non-nilpotent. Engel's theorem then allows us to choose some $x\in\g$ such that $\operatorname{ad}_x$ is not nilpotent. By the discussion above its corresponding matrix in a real orthonormal basis is skew-hermitian. Its complexification is therefore diagonalizable with purely imaginary eigenvalues~\cite{Meyer:2001}. Since $\operatorname{ad}_x:\g\to\g$ is not nilpotent, $\operatorname{ad}_x:\g_\C\to\g_\C$ is also not nilpotent and has therefore a non-zero eigenvalue $\lambda$~\cite{Fischer:2005}. Furthermore, because $\operatorname{ad}_x$ corresponds to a real matrix, it is also skew-symmetric and possesses therefore also $\lambda^*=-\lambda\neq\lambda$ as a non-zero eigenvalue~\cite{Drazin:1952}. The extended adjoint map $\operatorname{ad}_x:\g_\C\to\g_\C$ therefore has a pair of non-zero eigenvalues such that their ratio is $-1$, which contradicts Lemma~\ref{lem:adx:positive:rational:ratios}, as $\g_\C$ is solvable according to Proposition~77 in \cite{Heib:2025:A1}. Thus no finite-dimensional non-zero solvable non-nilpotent Lie subalgebra of $\hat{\mathcal{U}}$ exists.

    Next, let $\g\subseteq\hat{\mathcal{U}}$ be a non-zero finite-dimensional nilpotent Lie subalgebra. Then $\g_\C$ is also nilpotent, as complexification is compatible with the lower central series~\cite{Heib:2025:A1}, and, due to Lemma~\ref{lem:nilpotent:U:subalgebras}, abelian. Hence, $\g$ is abelian and because all real non-zero finite-dimensional abelian Lie algebras are isomorphic to $\R^n$ for some $n\in\N_{>0}$, isomorphic to such a Lie algebra. Thus, every non-zero finite-dimensional solvable Lie subalgebra of $\hat{\mathcal{U}}$ is abelian and isomorphic to  $\R^n$ for some $n\in\N_{>0}$.

    Finally, let us consider the case that $\g\subseteq\hat{\mathcal{U}}$ is a non-zero non-solvable Lie subalgebra. Here, we first recall that complexification is compatible with the derived series. That is, $\mathcal{D}\g_\C=(\mathcal{D}\g)_\C$. Thus, by Theorem~\ref{thm:classification:Usl2C}, one has $\g_\C=\mathcal{D}\g_\C\oplus\mathcal{Z}(\g_\C)$ with $\mathcal{D}\g_\C\cong\sl{2}{\C}$. Since $\g=\{v\in\g_\C\,\mid\,v^\dagger=-v\}$, one has consequently also $\g=\mathcal{D}\g\oplus\mathcal{Z}(\g)$, since both summands are preserved under conjugation with respect to $\g$. This implies, in particular, that $\g$ is reductive and it remains to identify the correct real form of $\sl{2}{\C}$ that $\mathcal{D}\g$ can be isomorphic to. 

    We may now express the Killing form $K(a,b)$ using the induced Frobenius inner product as follows. First pick an orthonormal basis $\mathcal{B}=\{c_1,c_2,c_3\}$ of $\mathcal{D}\g$ and then set
    \begin{align*}
        K:\mathcal{D}\g\times\mathcal{D}\g\longrightarrow\R,\quad(a,b)\mapsto K(a,b):=\Tr{\operatorname{ad}_a\circ\operatorname{ad}_b}:=\sum_{j=1}^3\iF{[a,[b,c_j]]}{c_j}.
    \end{align*}
    The induced Frobenius inner product is associative, thus
    \begin{align*}
        K(x,x)=\sum_{j=1}^3\iF{[x,[x,c_j]]}{c_j}=-\sum_{j=1}^3\iF{[[x,c_j],x]}{c_j}=-\sum_{j=1}^3\iF{[x,c_j]}{[x,c_j]},
    \end{align*}
    for all $x\in \mathcal{D}\g$. Since $(\mathcal{D}\g)_\C$ is simple, $\mathcal{D}\g$ is also simple and hence $\mathcal{Z}(\mathcal{D}\g)=0$. This implies that $[x,c_j]=0$ for all $j\in\{1,2,3\}$ only if $x=0$. The positive definiteness of the induced Frobenius inner product then implies that the Killing form for $\mathcal{D}\g$ is negative definite, which implies that $\mathcal{D}\g$ is compact~\cite{Knapp:1996}. It follows that $\mathcal{D}\g\cong\mathfrak{su}_\R(2)$~\cite{Bianchi:1903}, which completes the proof, since the reductiveness of $\g$ yields now $\g\cong\mathfrak{su}_\R(2)\oplus\R^n$ for some $n\in\Np{}$.
\end{proof}

\begin{proposition}
    Every finite-dimensional non-zero Lie algebra $\g$ from Corollary~\ref{thm:classification:hat:Usl2C} appears as a Lie subalgebra of $\hat{\mathcal{U}}$.
\end{proposition}

\begin{proof}
    It holds that $\mathrm{i}C^k$ is skew-hermitian, in the sense defined for $\hat{\mathcal{U}}$, for every $k\in\Np{}$, because $C^\dagger=C$. The PBW theorem guarantees that different  powers $C^k$ are linearly independent, as they have distinct degrees. Thus, $\spn_\R\{\mathrm{i}C^k\,\mid\,k\in\{1,\ldots,n\}\}\cong\R^n$. Also, $\mathfrak{su}_\R(2)\cong\spn_\R\{X,Y,Z\}$. Since the basis elements $\mathrm{i}C^k$ are central and have even positive degrees, their span has only trivial intersection with $\mathfrak{su}_\R(2)\cong\spn_\R\{X,Y,Z\}$. The sum of both spaces is therefore direct and realizes the Lie subalgebra $\su_\R(2)\oplus\R^n$ for every $n\in\Np{}$.
\end{proof}

Finally, we conjecture the following.

\begin{conjecture}
    Let $\varphi:\su_\R(2)\hookrightarrow\hat{\mathcal{U}}$ be an embedding. Then $\varphi(\su_\R(2))=\spn_\R\{X,Y,Z\}$. That is, it holds that $\varphi=\iota\circ\psi$, where $\iota$ denotes the canonical embedding and $\psi\in\operatorname{Aut}(\su_\R(2))$.
\end{conjecture}

It should be remarked that equipping $\mathcal{U}$ with a different conjugate-linear anti-involution $(\cdot)^{\Tilde{\dagger}}:\mathcal{U}\to\mathcal{U}$ can result in different real \enquote{skew-hermitian} Lie subalgebras. For example, set $H^{\Tilde{\dagger}}:=-H$, $E^{\Tilde{\dagger}}:=-E$, $F^{\Tilde{\dagger}}:=-F$, and $\unity^{\Tilde{\dagger}}:=\unity$, extended conjugate-linearly and by reversing products. These rules preserve the defining relations and allow the real Lie subalgebra $\sl{2}{\R}\cong\spn_\R\{H,F,E\}$ inside the resulting real \enquote{skew-hermitian} component of $\mathcal{U}$. Thus the compact classification above depends on the chosen involution. We chose $X^\dagger:=-X$, $Y^\dagger:=-Y$, and $Z^\dagger:=-Z$, motivated by their skew-hermitian Jordan--Schwinger realizations~\cite{Heib:2026:SSRC:Universality,Jordan:1935,Schwinger1960}.

\section{Discussion and Applications}\label{sec:discussion:applications}

The need to classify finite-dimensional Lie subalgebras of $\mathcal{U}$ originally arose from the study of physically relevant finite-dimensional Lie subalgebras of the Weyl algebra, in particular in one and two modes~\cite{Heib:2025:A1,Bruschi:Xuereb:2024,Heib:2026:Main}. In the two-mode setting, one is particularly interested in the centralizer of the free Hamiltonian $H_0=\omega(a^\dagger a+ b^\dagger b)$, where $\omega>0$, $a\equiv a_1$ and $b\equiv a_2$ are the annihilation operators of two orthogonal modes, with creation operators $a^\dagger\equiv a_1^\dagger$ and $b^\dagger\equiv a_2^\dagger$, satisfying the canonical commutation relations $[a,a^\dagger]=\mathds{1}=[b,b^\dagger]$ and $[a,b]=0=[a, b^\dagger]$~\cite{Messiah:1999,Dixmier:1973}, where $\mathds{1}$ shall denote the associative unit in the Weyl algebra. Thus, the goal of this section is to relate the classification from Theorem~\ref{thm:classification:Usl2C} to explicit realizations of the corresponding Lie subalgebras within the Weyl algebras $A_1$ and $A_2$.

Nevertheless, before doing so, we want to make the following observation.

\begin{corollary}
    Every finite-dimensional Lie subalgebra of $\mathcal{U}$ is minimal-graph-admissible.
\end{corollary}

\begin{proof}
    This is a direct consequence of Theorem~\ref{thm:classification:Usl2C} and Definition~\ref{def:assocaited:graph}, together with the bases of the Lie algebras $\C^{d_0}$, $\sl{2}{\C}$, and $\mathfrak{r}(\boldsymbol{d},\boldsymbol{m})$, employed throughout this work.  The union of the chosen bases works for each direct sum, since all mixed brackets vanish, and the property is preserved under Lie-algebra isomorphisms~\cite{Heib:2026:Graphs}. 
\end{proof}

\subsection{Conventions}

The complex Weyl algebra $A_n$ is the unital associative algebra generated by the set $\{a_j,a_j^\dagger,\mathds{1}\}_{j=1}^n$, where the operators $a_j$ are denoted as annihilation operators and $a_j^\dagger$ as creation operators. These satisfy the canonical commutation relations $[a_j,a_k^\dagger]=\delta_{jk}\mathds{1}$, while all other commutators vanish, and we write $a^\gamma=a^{(\alpha,\beta)}=(a_1^\dagger)^{\alpha_1}\ldots (a_n^\dagger)^{\alpha_n}a_1^{\beta_1}\ldots a_n^{\beta_n}$, where $\gamma=(\alpha,\beta)\in\Np{2n}$, for the normal ordered monomial basis elements of $A_n$~\cite{Dixmier:1968,Bruschi:Xuereb:2024}. In the Weyl algebra, the degree of a non-zero polynomial $p=\sum_{\gamma\in \Gamma} c_\gamma a^\gamma\in A_n$, with $c_\gamma\neq 0$ for all $\gamma$ in the finite set $\Gamma$, is defined as $\deg(p)=\max\{|\gamma|\,\mid\, \gamma\in \Gamma\}$, where $|\gamma|=|\alpha|+|\beta|=\alpha_1+\ldots+\alpha_n+\beta_1+\ldots+\beta_n$.

The conjugate-linear anti-involution $(\cdot)^\dagger: A_n\to A_n$ is defined via $(a_j)^\dagger:=a_j^\dagger$, $(a_j^\dagger)^\dagger:=a_j$, $\mathds{1}^\dagger=\mathds{1}$, $(\mathrm{i}\mathds{1})^\dagger=-\mathrm{i}\mathds{1}$, and the rule $(AB)^\dagger=B^\dagger A^\dagger$~\cite{Heib:2025:A1}. The real skew-hermitian Weyl algebra is then defined as the space
\begin{align*}
    \hat{A}_n:=\{p\in A_n\,\mid\, p^\dagger=-p\},
\end{align*}
equipped with the commutator bracket inherited from $A_n$~\cite{Bruschi:Xuereb:2024}.

\subsection{Embedding into the single-mode Weyl algebra $A_1$}\label{subsec:embedding:single:mode:Weyl:algebra}

It is well known that the special linear Lie algebra $\sl{2}{\C}$ embeds into the single-mode Weyl algebra $A_1$~\cite{TST:2006}, for instance, by utilizing the faithful realization $\varphi:\sl{2}{\C}\hookrightarrow A_1$ defined by
\begin{align*}
    \varphi(X)&:=\frac{e^{\frac{\mathrm{i}\pi}{4}}}{4}\left(a^2+\mathrm{i}(a^\dagger)^2\right)+\frac{\mathrm{i}}{\sqrt{8}}\left(2a^\dagger a+\mathds{1}\right),\;&\;\varphi(Y)&:=\frac{e^{\frac{\mathrm{i}\pi}{4}}}{4}\left(-a^2+\mathrm{i}(a^\dagger)^2\right),\;&\;\varphi(Z)&:=\frac{e^{\frac{\mathrm{i}\pi}{4}}}{\sqrt{8}}\left(\mathrm{i}a^2-(a^\dagger)^2\right)-\frac{1}{4}\left(2a^\dagger a+\mathds{1}\right)
\end{align*}
or equivalently, via
\begin{align}\label{eqn:realization:HEF:in:hat:A:1}
    \varphi(E)&:=\frac{\mathrm{i}}{\sqrt{8}}\left(a^2+(a^\dagger)^2+g_0\right),\;&\;\varphi(F)&:=-\frac{1}{\sqrt{8}}\left(a^2-(a^\dagger)^2+\mathrm{i}g_0\right),\;&\;\varphi(H)&:=\frac{e^{\frac{\mathrm{i}\pi}{4}}}{\sqrt{2}}\left(a^2+\mathrm{i}(a^\dagger)^2\right)+\frac{\mathrm{i}}{2}g_0,
\end{align}
where $g_0:=2a^\dagger a+\mathds{1}=a^\dagger a+aa^\dagger$.
These satisfy the required commutation relations~\cite{Heib:2025:A1}. Moreover, their normal ordered expansions are linearly independent, verifying that $\varphi$ is a faithful Lie-algebra homomorphism. Since $\varphi(E)$, $\varphi(F)$, and $\varphi(H)$ are all skew-hermitian in $A_1$, the restriction of  $\varphi$ to $\spn_\R\{E,F,H\}\cong\sl{2}{\R}$ constitutes an embedding of $\sl{2}{\R}$ into $\hat{A}_1$. Furthermore, it is also known that the only real form of $\sl{2}{\C}$ that embeds into the real skew-hermitian Weyl algebra $\hat{A}_1$ is $\sl{2}{\R}$~\cite{Heib:2025:A1}. 

The realization above is quadratic in the creation and annihilation operators. However, we can employ the construction in Proposition~\ref{eqn:non:linear:realizations:sl2C}, with suitable modifications regarding additional phase factors to preserve skew-hermiticity, to obtain an embedding of $\sl{2}{\R}$ into $\hat{A}_1$ that realizes, for every $m\in\N_{>1}$, at least one basis element of degree $4m-2$ in the creation and annihilation operators. To be precise, one has $\psi_m:\sl{2}{\R}\hookrightarrow \hat{A}_1$, by defining
\begin{align}\label{eqn:non:linear:realizations:sl2C:in:A1}
    \psi_m(E)&=\frac{\mathrm{i}}{\sqrt{8}}\left(a+a^\dagger-\left(e^{-\frac{\mathrm{i}\pi}{4}}a+e^{\frac{\mathrm{i}\pi}{4}}a^\dagger\right)^{2m-1}\right)^2,\;&\;
    \psi_m(F)&=\varphi(F),\;&\; 
    \psi_m(H)&=\varphi(H)-\frac{\mathrm{i}}{\sqrt{2}}\left(e^{-\frac{\mathrm{i}\pi}{4}}a+e^{\frac{\mathrm{i}\pi}{4}}a^\dagger\right)^{2m}.
\end{align}
For a detailed derivation, we refer to Appendix~\ref{app:realization:calcs}. This illustrates that finite-dimensional Lie algebras can be generated by nonlinear elements of a skew-hermitian Weyl algebra, that is, polynomials in bosonic creation and annihilation operators of total degree greater than two~\cite{Braunstein:2005}. Related Lie-algebraic structures also occur in quasi-exactly solvable quantum problems~\cite{Turbiner:1988,GonzalezLopez:1991}. For bosonic systems whose dynamical Lie algebra is finite-dimensional, the Wei--Norman method (cf. Ref.~\cite{Wei:Norman:1963,Wei:Norman:1964}) provides, locally in time and for suitable representations that integrate to unitary Lie-group representations, an exact factorization of the physically interesting time-evolution operator into finitely many exponentials~\cite{Bruschi:Xuereb:2024,Heib:2025:A1}. The time evolution of a system whose time-dependent Hamiltonian, describing the physical system of interest, generates such a finite-dimensional dynamical Lie algebra is therefore determined by a finite system of generally nonlinear differential equations, even when the underlying Hilbert space is infinite-dimensional. Thus, this method provides an exact alternative to truncating the Dyson series~\cite{Dyson:1949}, the Magnus expansion~\cite{Magnus:1954,Blanes:2009}, or the BCH-series~\cite{Rossmann:2002}, and to finite product approximations based on Lie-Suzuki-Trotter  formulas~\cite{Hall:Lie:groups:16,Trotter:1959,Suzuki:1976,Suzuki:1991}. Although the resulting differential equations need not admit closed-form solutions, the factorization itself requires no such truncation.

Finally, we remark that $\mathcal{U}$ itself does not embed into $A_1$ as an associative algebra, which is also a Lie algebra with the commutator bracket, since $\dim(\sl{2}{\C})=3>2$. The exclusion is shown in the lemma below.

\begin{lemma}
    Let $\g$ be a non-zero finite-dimensional Lie algebra with dimension $\dim(\g)=m$. Then, a unital associative-algebra embedding $U(\g)\hookrightarrow A_n$, which is also a Lie-algebra embedding for the commutator brackets, can exist only if $m\leq 2n$.
\end{lemma}

\begin{proof}
    Let $\mathcal{B}:=\{b_1,\ldots,b_m\}$ be a basis of $\g$ and suppose $\varphi:U(\g)\hookrightarrow A_n$. Set $d:=\max\{\deg(\varphi(b_j))\,\mid\, j\in\{1,\ldots,m\}\}$. For fixed $N\in\Np{}$, it holds now that
    \begin{align*}
        \dim\left(\varphi\left(\spn\left\{b_1^{k_1}\ldots b_m^{k_m}\,\middle\mid\,(k_1,\ldots,k_m)\in\mathcal{I}_m^{(N)}\right\}\right)\right)=\dim\left(\spn\left\{b_1^{k_1}\ldots b_m^{k_m}\,\middle\mid\,(k_1,\ldots,k_m)\in\mathcal{I}_m^{(N)}\right\}\right)=\binom{N+m}{m}.
    \end{align*}
    However, since $\varphi(\spn\{b_1^{k_1}\ldots b_m^{k_m}\,\mid\,(k_1,\ldots,k_m)\in\mathcal{I}_m^{(N)}\})\subseteq\spn\{a^\gamma\,\mid\,\gamma\in\Np{2n}:\,|\gamma|\leq dN\}$, it must also hold that
    \begin{align*}
        \binom{N+m}{m}\leq\dim\left(\spn\left\{a^\gamma\,\mid\,\gamma\in\Np{2n}:\,|\gamma|\leq dN\right\}\right)=\binom{dN+2n}{2n}.
    \end{align*}
    This inequality can clearly only hold for all $N\in\Np{}$ if $m\leq 2n$.
\end{proof}

However, it is easy to recognize that the embedding $\varphi:\sl{2}{\C}\hookrightarrow A_1$ gives rise to a unital associative-algebra homomorphism $\Phi:\mathcal{U}\to A_1$, which is also a Lie-algebra homomorphism for commutator brackets, where $\ker(\Phi)=(C-3\unity/16)\mathcal{U}$. This is straightforward to verify. Indeed, explicit calculations reveal that $\Phi(C)=3\mathds{1}/16$. Thus, it holds that $(C-3\unity/16)\mathcal{U}\subseteq \ker(\Phi)$. To obtain the converse, recall that $(C-3\unity/16)\mathcal{U}$ is a maximal two-sided ideal of $\mathcal{U}$~\cite{Arnal:1973}. Moreover, $\ker(\Phi)$ is a proper ideal since $\Phi(0)=0$ and $\Phi(\unity)=\mathds{1}$. A maximal proper ideal cannot be strictly contained in another proper ideal. Thus, one must have $\ker(\Phi)=(C-3\unity/16)\mathcal{U}$.

Consequently, the representatives constructed for Theorem~\ref{thm:classification:Usl2C} have the following image Lie algebras: $\C^{d_0}$, $\sl{2}{\C}$, $\sl{2}{\C}\oplus\C$, $\mathfrak{r}(\boldsymbol{d},\boldsymbol{m})$, and $\mathfrak{r}(\boldsymbol{d},\boldsymbol{m})\oplus\C$, where $\boldsymbol{m}=(1,\ldots,1)$, since powers of $C$ become scalar. These have all been identified as finite-dimensional Lie subalgebras of the complex single-mode Weyl algebra $A_1$~\cite{TST:2006}. Moreover, all of these also occur with at least one real form within the real skew-hermitian Weyl algebra~\cite{Heib:2025:A1}.

\subsection{Embedding into the two-mode Weyl algebra $A_2$}\label{subsec:embedding:two:mode:Weyl:algebra}

The embedding of $\sl{2}{\C}$ into the two-mode Weyl algebra $A_2$ is known as the Jordan-Schwinger map, which, in the convention of Lemma~\ref{lem:faithful:skew:hermitian:representation}, can be defined via~\cite{Heib:2026:SSRC:Universality,Schwinger:1952,Jordan:1935}
\begin{align*}
    \varphi(X)&=-\frac{\mathrm{i}}{2}\left(a_1^\dagger a_2+a_2^\dagger a_1\right),\;&\;\varphi(Y)&:=\frac{1}{2}\left(a_1^\dagger a_2-a_2^\dagger a_1\right),\;&\;\varphi(Z)&:=\frac{\mathrm{i}}{2}\left(a_1^\dagger a_1-a_2^\dagger a_2\right).
\end{align*}
These images are clearly linearly independent, skew-hermitian, and satisfy the cyclic commutation relations of the basis elements of $\su_\R(2)$~\cite{Pfeifer:2003}. Thus, $\varphi$ restricts to an embedding $\su_\R(2)\hookrightarrow \hat{A}_2$. The induced unital associative-algebra homomorphism $\Phi:\mathcal{U}\to A_2$, which is also a Lie-algebra homomorphism for the commutator brackets, is furthermore injective: under the identification $a_j^\dagger\leftrightarrow x_j$ and $a_j\leftrightarrow \partial_j$, one observes that $\Phi(\mathcal{U})\cong P(\mathcal{U})\cong \mathcal{U}$, where $P$ is the faithful polynomial representation of $\mathcal{U}$ from the proof of Lemma~\ref{lem:faithful:skew:hermitian:representation}.

\section{Conclusion}\label{sec:conclusion}

In this work, we have classified the non-zero finite-dimensional complex Lie subalgebras of the universal enveloping algebra $\mathcal{U}$ of the special linear Lie algebra $\sl{2}{\C}$ up to isomorphism. Moreover, we have provided explicit realizations of the listed types. In a second step, we equipped $\mathcal{U}$ with a conjugate-linear anti-involution $(\cdot)^\dagger$ and classified the real finite-dimensional Lie subalgebras of its real infinite-dimensional skew-hermitian part $\hat{\mathcal{U}}$, consisting of all elements satisfying $A=-A^\dagger$. We then discussed the consequences of these results for embeddings of such finite-dimensional Lie algebras in the single- and two-mode Weyl algebras. 

Potential future research directions include the classification of finite-dimensional Lie subalgebras of $U(\mathfrak{sp}(2\cdot 2,\C))$ or $U(\mathfrak{sp}(2\cdot 2,\C)\ltimes\mathfrak{h}_2)$, where $\mathfrak{h}_2$ denotes the five-dimensional complex Heisenberg algebra. The canonical real forms of the symplectic Lie algebra $\mathfrak{sp}(2\cdot 2,\C)$ and the associated Jacobi algebra $\mathfrak{sp}(2\cdot 2,\C)\ltimes\mathfrak{h}_2$ describe Gaussian unitary transformations of two bosonic modes, without and with displacements, respectively~\cite{Weedbrook:2012,Berceanu:2007}. The complex universal enveloping algebras above may therefore be of interest for studying nonlinear generalizations~\cite{Karassiov:1991,Karassiov:1992,Karassiov:1998}. Other future research directions concern classifying the finite-dimensional Lie subalgebras of genuine multi-mode Weyl algebras and their skew-hermitian counterparts.

\section*{Data availability statement}
	
Data sharing is not applicable to this work as no datasets are used to support the presented results.

\section*{Acknowledgments}

The author thanks David Edward Bruschi, Niklas Jung, and Euan Walker for listening to his presentations of parts of the proofs. He is especially grateful to David Edward Bruschi for pointing out an error in one of his proofs, as well as for his guidance and insightful comments.

\section*{Conflict of interest}
	
The author has no relevant financial or non-financial interests to disclose.

\section*{Author contributions}

Tim Heib developed the mathematical ideas and proofs presented in this work. He also produced all figures.\par\noindent
\textit{Use of a large language model.} Tim Heib used an LLM to assist with reviewing proofs for possible errors, checking spelling and grammar, conducting literature searches, and editing the bibliography. He critically assessed the suggestions and independently checked every cited reference. Responsibility for the final content remains with the author.

\section*{Funding}
    
Tim Heib received no external funding for this work.

\bibliography{main}

\onecolumngrid
\newpage
\appendix

\section{On the realizations of $\sl{2}{\R}$ in the real skew-hermitian Weyl algebra $\hat{A}_1$}\label{app:realization:calcs}

In this section, we want to show that the realization defined in Eq.~\ref{eqn:non:linear:realizations:sl2C:in:A1} of $\sl{2}{\R}$ in $\hat{A}_1$ is indeed obtained by the prescribed construction.

First, let us recall the construction employed in Proposition~\ref{eqn:non:linear:realizations:sl2C}. For $m\in\N_{>1}$, utilizing the canonical basis elements $H,E,F\in\mathcal{U}$ satisfying the Lie bracket relations $[H,E]=2E$, $[H,F]=-2F$, and $[E,F]=H$, we can construct an embedding of $\sl{2}{\C}$ into $\mathcal{U}$ by defining the linear map $\chi_m:\sl{2}{\C}\to \mathcal{U}$ on the basis elements $H,E,F$ via
\begin{align*}
    \chi_m(H)&:=H-F^m,\;&\;\chi_m(E)&:=E+\frac{1}{2m}\sum_{\ell=0}^{m-1}F^\ell H F^{m-1-\ell}-\frac{1}{4}F^{2m-1},\;&\;\chi_m(F)&:= F.
\end{align*}
Here, we may now recognize that
\begin{align*}
    \chi_m(H)&=H-\frac{1}{2m}[F^m,H]=\operatorname{ad}_{-\frac{1}{2m}F^m}^0(H)+\operatorname{ad}_{-\frac{1}{2m}F^m}^1(H),\\
    \chi_m(E)&=E-\frac{1}{2m}[F^m,E]+\frac{1}{8m^2}[F^m,[F^m,E]]=\operatorname{ad}_{-\frac{1}{2m}F^m}^0(E)+\operatorname{ad}_{-\frac{1}{2m}F^m}^1(E)+\frac{1}{2}\operatorname{ad}_{-\frac{1}{2m}F^m}^2(E).
\end{align*}
Since $F$ commutes with any power of itself, it holds consequently, for all $B\in\sl{2}{\C}$, that
\begin{align*}
    \chi_m(B)=\exp\left(\operatorname{ad}_{-\frac{1}{2m}F^m}\right)(B)=\exp\left(-\frac{1}{2m}\operatorname{ad}_{F^m}\right)(B).
\end{align*}
In order to construct an embedding of $\sl{2}{\R}$ into $\hat{A}_1$, we introduce a yet undetermined complex factor $\theta_m\in\C$ into the exponential:
\begin{align*}
    \Tilde{\chi}_m:\sl{2}{\C}\longrightarrow \mathcal{U},\quad B\mapsto \Tilde{\chi}_m(B):=\exp\left(-\frac{\theta_m}{2m}\operatorname{ad}_{F^m}\right)(B).
\end{align*}
Here it is essential to note that, for all $\kappa_m\in\C$, the derivation $\operatorname{ad}_{\kappa_m F^m}$ is locally nilpotent on $\mathcal{U}$~\cite{Daigle:2019}. Indeed, it holds that $\operatorname{ad}_{F^m}^3(A)=0$ for all $A\in\spn\{H,E,F\}$, which implies that, for each $B\in\mathcal{U}$, there exists some $\ell\in\Np{}$ such that $\operatorname{ad}_{\kappa_m F^m}^\ell(B)=0$. Thus, for all $\theta_m\in\C$, the map $\Tilde{X}_m:\mathcal{U}\to \mathcal{U}$, obtained by extending $\Tilde{\chi}_m$ uniquely using the universal property of $\mathcal{U}$, is a unital associative-algebra automorphism and hence also a Lie-algebra automorphism of $\mathcal{U}$~\cite{Humphreys:1972}. Its restriction $\Tilde{\chi}_m$ is therefore a Lie-algebra embedding of $\sl{2}{\C}$ into $\mathcal{U}$, preserving the commutation relations.

In the next step, we recognize that, when writing out the expressions for the realizations from Eq.~\eqref{eqn:realization:HEF:in:hat:A:1} without normal ordering, one finds
\begin{align*}
    \varphi(E)&=\frac{\mathrm{i}}{\sqrt{8}}\left(a+a^\dagger\right)^2,\;&\;\varphi(F)&=-\frac{1}{\sqrt{8}}\left(a+\mathrm{i}a^\dagger\right)^2,\;&\;\varphi(H)&=\frac{1+\mathrm{i}}{4}\left((a+a^\dagger)(a+\mathrm{i}a^\dagger)+(a+\mathrm{i}a^\dagger)(a+a^\dagger)\right).
\end{align*}
Here, it is helpful to notice that $[a+\mathrm{i}a^\dagger, a+a^\dagger]=(1-\mathrm{i})\mathds{1}$. We may then compute:
\begin{align*}
    \left[-\frac{\theta_m}{2m}(\varphi(F))^m,\varphi(E)\right]&=-\frac{\theta_m}{2m}\left(-\frac{1}{\sqrt{8}}\right)^m\frac{\mathrm{i}}{\sqrt{8}}\left[\left(a+\mathrm{i}a^\dagger\right)^{2m},\left(a+a^\dagger\right)^2\right]\\
    &=-\frac{\theta_m}{2m}\left(-\frac{1}{\sqrt{8}}\right)^m\frac{\mathrm{i}}{\sqrt{8}}2m(1-\mathrm{i})\left(\left(a+a^\dagger\right)\left(a+\mathrm{i}a^\dagger\right)^{2m-1}+\left(a+\mathrm{i}a^\dagger\right)^{2m-1}\left(a+a^\dagger\right)\right)\\
    &=(1+\mathrm{i})\theta_m\left(-\frac{1}{\sqrt{8}}\right)^{m+1}\left(\left(a+a^\dagger\right)\left(a+\mathrm{i}a^\dagger\right)^{2m-1}+\left(a+\mathrm{i}a^\dagger\right)^{2m-1}\left(a+a^\dagger\right)\right),
\end{align*}
\begin{align*}
    \frac{1}{2}\left[-\frac{\theta_m}{2m}(\varphi(F))^m,\left[-\frac{\theta_m}{2m}(\varphi(F))^m,\varphi(E)\right]\right]&=-(1+\mathrm{i})\frac{\theta_m^2}{4m}\left(-\frac{1}{\sqrt{8}}\right)^{2m+1}4m(1-\mathrm{i})\left(a+\mathrm{i}a^\dagger\right)^{4m-2}\\
    &=-2\theta_m^2\left(-\frac{1}{\sqrt{8}}\right)^{2m+1}\left(a+\mathrm{i}a^\dagger\right)^{4m-2},
\end{align*}
and
\begin{align*}
     \left[-\frac{\theta_m}{2m}(\varphi(F))^m,\varphi(H)\right]&=-\frac{\theta_m}{2m}\frac{1+\mathrm{i}}{4}\left(-\frac{1}{\sqrt{8}}\right)^m4m(1-\mathrm{i})\left(a+\mathrm{i}a^\dagger\right)^{2m}=-\theta_m\left(-\frac{1}{\sqrt{8}}\right)^m\left(a+\mathrm{i}a^\dagger\right)^{2m}.
\end{align*}
Hence, the composition of $\Tilde{\chi}_m$ with the extension $\Phi:\mathcal{U}\to A_1$ of $\varphi$ to $\mathcal{U}$ yields the Lie-algebra homomorphism $\psi_m:=\Phi\circ\Tilde{\chi}_m$, with
\begin{align*}
    \psi_m(E)&=\frac{\mathrm{i}}{\sqrt{8}}\left(a+a^\dagger-(1-\mathrm{i})\theta_m\left(-\frac{1}{\sqrt{8}}\right)^m\left(a+\mathrm{i}a^\dagger\right)^{2m-1}\right)^2,\\
    \psi_m(F)&=-\frac{1}{\sqrt{8}}\left(a+\mathrm{i}a^\dagger\right)^2,\\
    \psi_m(H)&=\frac{1+\mathrm{i}}{4}\left((a+a^\dagger)(a+\mathrm{i}a^\dagger)+(a+\mathrm{i}a^\dagger)(a+a^\dagger)\right)-\theta_m\left(-\frac{1}{\sqrt{8}}\right)^m\left(a+\mathrm{i}a^\dagger\right)^{2m}.
\end{align*}
Finally, we need to choose $\theta_m$ such that $\psi_m(E)$, $\psi_m(F)$, and $\psi_m(H)$ are all skew-hermitian within $A_1$. Currently, it holds that
\begin{align*}
    \left(\psi_m(F)\right)^\dagger&=-\frac{1}{\sqrt{8}}\left(a^2+\mathrm{i}\left(a^\dagger a+aa^\dagger\right)-(a^\dagger)^2\right)^\dagger=-\frac{1}{\sqrt{8}}\left((a^\dagger)^2-\mathrm{i}\left(a^\dagger a+aa^\dagger\right)-a^2\right)=\frac{1}{\sqrt{8}}\left(a^2+\mathrm{i}\left(a^\dagger a+aa^\dagger\right)-(a^\dagger)^2\right)\\
    &=-\psi_m(F),
\end{align*}
which suffices the condition $\psi_m(F)\in\hat{A}_1$. Considering $\psi_m(H)$, on the other hand, yields
\begin{align*}
    \left(\psi_m(H)\right)^\dagger&=\frac{1-\mathrm{i}}{4}\left((a^\dagger-\mathrm{i}a)(a+a^\dagger)+(a+a^\dagger)(a^\dagger-\mathrm{i}a)\right)-\theta_m^*\left(-\frac{1}{\sqrt{8}}\right)^m(-1)^m(a+\mathrm{i}a^\dagger)^{2m}\\
    &=-\mathrm{i}\frac{1-\mathrm{i}}{4}\left((a+a^\dagger)(a+\mathrm{i}a^\dagger)+(a+\mathrm{i}a^\dagger)(a+a^\dagger)\right)-(-1)^m\theta_m^*\left(-\frac{1}{\sqrt{8}}\right)^m(a+\mathrm{i}a^\dagger)^{2m}\\
    &=-\frac{1+\mathrm{i}}{4}\left((a+a^\dagger)(a+\mathrm{i}a^\dagger)+(a+\mathrm{i}a^\dagger)(a+a^\dagger)\right)-(-1)^m\theta_m^*\left(-\frac{1}{\sqrt{8}}\right)^m(a+\mathrm{i}a^\dagger)^{2m}.
\end{align*}
Thus, for $(\psi_m(H))^\dagger=-\psi_m(H)$ to hold, we need to require $\theta_m^*(-1)^m=-\theta_m$. This is consistent with the requirement stemming from $(\psi_m(E))^\dagger=-\psi_m(E)$, as
\begin{align*}
    \left(\psi_m(E)\right)^\dagger&=-\frac{\mathrm{i}}{\sqrt{8}}\left(a+a^\dagger-(1+\mathrm{i})\theta_m^*\left(-\frac{1}{\sqrt{8}}\right)^m(-1)^{m-1}(a^\dagger-\mathrm{i}a)(a+\mathrm{i}a^\dagger)^{2m-2}\right)^2\\
    &=-\frac{\mathrm{i}}{\sqrt{8}}\left(a+a^\dagger-\mathrm{i}(1+\mathrm{i})(-1)^m\theta_m^*\left(-\frac{1}{\sqrt{8}}\right)^m(a+\mathrm{i}a^\dagger)^{2m-1}\right)^2\\
    &=-\frac{\mathrm{i}}{\sqrt{8}}\left(a+a^\dagger-(1-\mathrm{i})\theta_m\left(-\frac{1}{\sqrt{8}}\right)^m(a+\mathrm{i}a^\dagger)^{2m-1}\right)^2=-\psi_m(E).
\end{align*}
The equation $\theta_m^*(-1)^m=-\theta_m$ is solved via $\theta_m=\mathrm{i}^{m+1}\rho_m$, where $\rho_m\in \R$ can be chosen freely. 
Since
\begin{align*}
    \left(e^{-\frac{\mathrm{i}\pi}{4}}a+e^{\frac{\mathrm{i}\pi}{4}}a^\dagger\right)^{2m-1}&=e^{-\frac{\mathrm{i}\pi(2m-1)}{4}}\left(a+e^{\frac{\mathrm{i}\pi}{2}}a^\dagger\right)^{2m-1}=\left(e^{-\frac{\mathrm{i}\pi}{2}}\right)^m e^{\frac{\mathrm{i}\pi}{4}}\left(a+\mathrm{i}a^\dagger\right)^{2m-1}=(-1)^m\mathrm{i}^m\frac{1}{\sqrt{2}}(1+\mathrm{i})\left(a+\mathrm{i}a^\dagger\right)^{2m-1},
\end{align*}
we may set $\rho_m=\sqrt{8^m/2}$, such that
\begin{align*}
    (1-\mathrm{i})\theta_m\left(-\frac{1}{\sqrt{8}}\right)^m\left(a+\mathrm{i}a^\dagger\right)^{2m-1}&=(1-\mathrm{i})\mathrm{i}^{m+1}\rho_m\left(-\frac{1}{\sqrt{8}}\right)^m\frac{\sqrt{2}}{(1+\mathrm{i})(-1)^m\mathrm{i}^m}\left(e^{-\frac{\mathrm{i}\pi}{4}}a+e^{\frac{\mathrm{i}\pi}{4}}a^\dagger\right)^{2m-1}\\
    &=(1+\mathrm{i})\mathrm{i}^m\rho_m\left(-\frac{1}{\sqrt{8}}\right)^m\frac{(-1)^m\sqrt{2}}{(1+\mathrm{i})\mathrm{i}^m}\left(e^{-\frac{\mathrm{i}\pi}{4}}a+e^{\frac{\mathrm{i}\pi}{4}}a^\dagger\right)^{2m-1}\\
    &=\sqrt{2}\rho_m\left(\frac{1}{\sqrt{8}}\right)^m\left(e^{-\frac{\mathrm{i}\pi}{4}}a+e^{\frac{\mathrm{i}\pi}{4}}a^\dagger\right)^{2m-1}\\
    &=\left(e^{-\frac{\mathrm{i}\pi}{4}}a+e^{\frac{\mathrm{i}\pi}{4}}a^\dagger\right)^{2m-1}.
\end{align*}
Then
\begin{align*}
    \psi_m(E)&=\frac{\mathrm{i}}{\sqrt{8}}\left(a+a^\dagger-\left(e^{-\frac{\mathrm{i}\pi}{4}}a+e^{\frac{\mathrm{i}\pi}{4}}a^\dagger\right)^{2m-1}\right)^2,\\
    \psi_m(F)&=-\frac{1}{\sqrt{8}}\left(a+\mathrm{i}a^\dagger\right)^2,\\
    \psi_m(H)&=\frac{1+\mathrm{i}}{4}\left((a+a^\dagger)(a+\mathrm{i}a^\dagger)+(a+\mathrm{i}a^\dagger)(a+a^\dagger)\right)-\frac{\mathrm{i}}{\sqrt{2}}\left(e^{-\frac{\mathrm{i}\pi}{4}}a+e^{\frac{\mathrm{i}\pi}{4}}a^\dagger\right)^{2m}.
\end{align*}
At last, we observe that the realization $\psi_m:\sl{2}{\R}\to \hat{A}_1$ is indeed faithful, since the three basis elements above are linearly independent, as seen by comparing the degrees in the creation and annihilation operators.

\end{document}